%% file: MAIN.tex
\documentclass[10pt,reqno]{amsart}

\newtheorem{theo}{Theorem}[part]

\input{PREAMBULE}

\begin{document}

\title{Bourgain's method for $K$-closedness in the semicommmutative setting}

\author{Hugues Moyart}
\address{UNICAEN, CNRS, LMNO, 14000 Caen, France}
\email{hugues.moyart@unicaen.fr}


\begin{abstract}
In the early 1990s, J.Bourgain proved a general result $K$-closedness result in the context of classical harmonic analysis. In this paper, we extend Bourgain's method to the semicommutative setting, making use of the recent semicommutative Calderón-Zygmund decomposition introduced by L.Cadilhac, JM.Conde-Alonso and J.Parcet. As an application, we recover Pisier's result about $K$-closedness of noncommutative Hardy spaces on the torus, and we also establish new interpolation results for noncommutative Sobolev spaces on the torus.
\end{abstract} 

\maketitle

\setcounter{tocdepth}{3}
\tableofcontents

\clearpage

\part*{Introduction}

The paper is motivated by Bourgain's approach to Jones' theorem in the context of interpolation of Hardy spaces on the torus. Let us review the results obtained in this context. Let $\TT$ denote the unit circle. For $1\pp p\pp\infty$, let $H_p(\T)$ denote the associated analytic Hardy space. Then, Jones established in \cite{JonesHardy2} that there is a universal constant $C$ such that for every $f\in H_1(\TT)+H_\infty(\TT)$, and $t>0$, we have 
\begin{equation}\label{strongJonesTheorem}
    K_t(f,H_1(\TT),H_\infty(\TT))\pp CK_t(f,L_1(\TT),L_\infty(\TT))
\end{equation}
where $K$ refers to Peetre's $K$-functional that is used to define real interpolation spaces. According to the terminology introduced by G.Pisier in \cite{PisierHardy}, one can reformulate Jones' theorem by saying that the subcouple $(H_1(\TT),H_\infty(\TT))$ is \textit{$K$-closed} (see Part 1) in the compatible couple $(L_1(\TT),L_\infty(\TT))$, with a universal constant. This $K$-closedness result completely characterizes the properties of analytic Hardy spaces under the real interpolation method. In particular, it implies that, if $0<\theta<1$ then
\begin{equation}\label{weakJonesTheorem}
    (H_1(\TT),H_\infty(\TT))_{\theta,p_\theta,K}=H_{p_\theta}(\TT)
\end{equation}
with equivalent norms, where $1/p_\theta=1-\theta$, and where the notation on the right-hand side refers to the real interpolation method. In \cite{Bourgain}, using the Calderón–Zygmund decomposition—the fundamental tool of harmonic analysis—and starting from a projection defined by a singular kernel satisfying the standard conditions, Bourgain proved a $K$-closedness result for the compatible couple $(H_1,H_2)$. By applying this result to the Riesz projection and its complement projection, he was then able to recover \eqref{weakJonesTheorem}. As highlighted by SV.Kislyakov and Q.Xu in \cite{KislyakovXu}, from Bourgain's approach one can actually deduce \eqref{strongJonesTheorem}. Bourgain's approach in this commutative setting later received further development in a number of works on interpolation of abstract Hardy-type spaces, for instance in the recent survey \cite{Rutsky}.

\vspace{5pt}

This paper is devoted to the noncommutative extension of Bourgain’s method. It essentially relies on the recent noncommutative generalization of the Calderón–Zygmund decomposition introduced by J. Parcet, J.M. Conde Alonso and L. Cadilhac in \cite{CadilhacCZ1,CadilhacCZ2} in the context of Euclidean spaces. In view of applying this method in other contexts (such as the torus), our results are stated in the framework of spaces of homogeneous type (and more precisely spaces of homogeneous type that satisfy an Ahlfors-type regularity condition). The details about the construction of the noncommutative Calderón–Zygmund decomposition are given in Part 4. The mathematical formalism of this article is that of abstract Hardy spaces, introduced in \cite{Moyart2026}. This framework is briefly described below. See Part 1 for details on the mathematical background.

\vspace{5pt}

If $M$ is a semifinite von Neumann algebra, $P$ is any projection on $L_2(M)$ and $1\pp p\pp\infty$, we define $H_p(P)$ to be the closure of $L_p(M)\cap P(L_2(M))$ in $L_p(M)$ (for the norm topology if $p<\infty$, and the w*-topology if $p=\infty$). They are the \textit{abstract Hardy spaces} associated with $P$. A natural question is whether abstract Hardy spaces form an interpolation scale for the real method. The main tool for addressing this problem is the notion of \textit{$K$-closedness}, introduced by Pisier in \cite{PisierHardy}, also referred to as \textit{quasi-complementation} in the terminology adopted in this thesis (and likewise called $K$-closedness in \cite{KislyakovXu}). The problem of proving a noncommutative analogue of Jones’ result for the compatible couple $(H_1,H_\infty)$—that is, showing that the subcouple $(H_1,H_\infty)$ is $K$-closed (or quasi-complemented, according to our terminology) in the compatible couple $(L_1,L_\infty)$—has triggered substantial developments in abstract interpolation theory. Relevant references include \cite{Bourgain,JansonRealInterpolation,KislyakovXu}. 

\vspace{5pt}

The main result of the paper is stated as follows. See Part 2 for the unexplained notions.

\begin{mdframed}[skipabove=10pt]
\textbf{\upshape Theorem} \text{\upshape (Bourgain's method).} \textit{Let $\X$ be a Ahlfors regular locally compact space, let $M$ be a semifinite von Neumann algebra, and let $N:=L_\infty(\X)\bar{\otimes}M$ be the tensor product von Neumann algebra equipped with the tensor product trace. Let $P$ be a projection on $L_2(N)$ which satisfies the technical assumptions, and which is given by a singular kernel on $\X$. Then the subcouple $(H_1(P),H_\infty(P))$ is quasi-complemented in $(L_1(M),L_\infty(M))$.}
\end{mdframed}
\vspace{15pt}

In Part 3, we build on the results of the previous section to develop noncommutative harmonic analysis on the torus. Our approach differs from classical methods, which typically rely on transference principles from Euclidean spaces. Instead, we work directly on the torus, making essential use of Bourgain's method.

\vspace{5pt}

Let $d\pg1$ be a positive integer, let $\T^d$ denote the $d$-torus, let $M$ be a semifinite von Neumann algebra, and let $N:=L_\infty(\T^d)\bar{\otimes}M$ denote the tensor product von Neumann algebra equipped with the tensor product trace.

\vspace{5pt}

For $1\pp p\pp\infty$, we define the \textit{analytic Hardy space} $H_p(N)$ as the space of $f\in L_p(N)$ such that
\[\int_{\T^d}f(z)z^{n}dz=0,\ \ \text{for every}\ \ n\in\Z_+^d.\]
Then, in the case $d=1$, by applying our noncommutative version of Bourgain's method, we provide a new proof of the following well-known theorem, independent of any complex-variable techniques. 

\begin{mdframed}[skipabove=10pt]
\textbf{\upshape Theorem} \text{\upshape (Pisier, Xu \cite{PisierXuLp}).} \textit{Assume $d=1$. The subcouple $(H_1(N),H_\infty(N))$ is quasi\-/complemented in the compatible couple $(L_1(N),L_\infty(N))$ with a universal constant, and for every $1<p<\infty$ we have
\[(H_1(N),H_\infty(N))_{\theta,p,K}=H_p(N)\]
with equivalent norms, with constants depending on $p$ only, where $\theta:=1-1/p$.}
\end{mdframed}
\vspace{15pt}

If $1\pp p\pp\infty$, we define the \textit{(homogeneous) Sobolev space} $W_{1,p}(N)$ as the closure in $L_p(N)$ (for the norm topology if $p<\infty$ and the w*-topology if $p=\infty)$ of trigonometric polynomials $g\in L_p(N)$ under the (semi)norm 
\[\|g\|_{W_{1,p}(N)}=\left\{\begin{array}{cl}
     \Big[\sum_{j=1}^{d}\|\partial_jg\|_{L_p(N)}^p\Big]^{1/p} & \text{if}\ p<\infty \\
     \sup_{j\in\{1,\ldots,d\}}\|\partial_jg\|_{L_\infty(N)} & \text{if}\ p=\infty
\end{array}\right..\]
As already noticed in \cite{Bourgain}, Bourgain's method can be applied in this context. With our noncommutative version of Bourgain's method, we obtain the following result, which provides an operator-valued extension of a classical result originally due to DeVore and Scherer \cite{DevoreInterpolationSobolev} in the context of inhomogeneous Sobolev spaces on Euclidean spaces.  

\begin{mdframed}[skipabove=10pt]
\textbf{\upshape Theorem.} \textit{
Let $f\in W_{1,1}(N)+W_{1,\infty}(N)$. Then, for every $t>0$, up to a constant depending only on $d$, we have
\[K_t(f,W_{1,1}(N),W_{1,\infty}(N))\ \simeq\ \sum_{j=1}^{d}K_t(\partial_jf,L_1(N),L_{\infty}(N)).\]}
\end{mdframed}

\vspace{5pt}

\begin{mdframed}[skipabove=10pt]
\textbf{\upshape Corollary.} \textit{
If $1\pp p\pp\infty$ then
\[W_{1,p}(N)=(W_{1,1}(N),W_{1,\infty}(N))_{\theta,p,K}\]
with equivalent norms, with constants depending on $p,d$ only, where $\theta:=1-1/p$.}
\end{mdframed}
\vspace{15pt}

For the complex method, we get the following result.

\begin{mdframed}[skipabove=10pt]
\textbf{\upshape Theorem.} \textit{
For every $1\pp p_0,p_1<\infty$ and $0<\theta<1$, we have
\[[W_{1,p_0}(N),W_{1,p_1}(N)]_{\theta}=W_{1,p_\theta}(N)\]
with equivalent norms, with constants depending only on $p_0,p_1,\theta,d$, where 
\[\frac{1}{p_\theta}=\frac{1-\theta}{p_0}+\frac{\theta}{p_1}.\]}
\end{mdframed}
\vspace{15pt}

The determination of complex interpolation spaces for the compatible couple $(W_{1,1}(N),W_{1,\infty}(N))$ remains open, even in the commutative case.

\clearpage

\part{Preliminaries}

\section{Abstract interpolation theory}

The material in this section is taken from \cite{JansonRealInterpolation} and \cite{BerghInterpolation}.

\subsection{Compatible couples}

A \textit{compatible couple} is a couple $(E_0,E_1)$ of subspaces of a common Hausdorff topological vector space $E$, such that $E_j$ is equipped with a complete norm that makes the inclusion $E_j\to E$ continuous, for $j\in\{0,1\}$. Then the intersection space $E_0\cap E_1$ and the sum space $E_0+E_1$ are canonically equipped with the complete norms $\|\cdot\|_{E_0\cap E_1}$ and $\|\cdot\|_{E_0+E_1}$ given by the expressions
\[\|u\|_{E_0\cap E_1}:=\max\big\{\|u\|_{E_0},\ \|u\|_{E_1}\big\},\]
\[\|u\|_{E_0+E_1}:=\inf\big\{\|u_0\|_{E_0}+\|u_1\|_{E_1}\ :\  u=u_0+u_1,\ u_0\in E_0,u_1\in E_1\big\}.\]
An \textit{intermediate space} for a compatible couple $(E_0,E_1)$ is a subspace $E_\theta$ of $E_0+E_1$ that contains $E_0\cap E_1$, and that is equipped with a complete norm that makes the inclusions $E_0\cap E_1\to E_\theta$ and $E_\theta\to E_0+E_1$ both continuous. If $E_{\theta_0}$, $E_{\theta_1}$ are intermediate spaces for a compatible couple $(E_0,E_1)$, then their sum $E_{\theta_0}+E_{\theta_1}$ and their intersection $E_{\theta_0}\cap E_{\theta_1}$ are also intermediate spaces for $(E_0,E_1)$ when equipped with the corresponding sum norm $\|\cdot\|_{E_{\theta_0}+E_{\theta_1}}$ and intersection norm $\|\cdot\|_{E_{\theta_0}\cap E_{\theta_1}}$ as defined above.

\subsection{Compatible bounded operators}

Let $(E_0,E_1)$ and $(F_0,F_1)$ be two compatible couples. A \textit{compatible bounded operator} $(E_0,E_1)\to(F_0,F_1)$ is an operator $T:E_0+E_1\to F_0+F_1$ such that, if $j\in\{0,1\}$, then $T$ maps $E_j$ into $F_j$, and $T:E_j\to F_j$ is bounded. In this situation, we set
\[\|T\|_{(E_0,E_1)\to(F_0,F_1)}:=\max\big\{\|T\|_{E_0\to F_0},\|T\|_{E_1\to F_1}\big\}.\]
Let $T:(E_0,E_1)\to(F_0,F_1)$ be a compatible bounded operator. Note that $T$ is injective (resp. surjective, bijective) if and only if $T:E_j\to F_j$ is, for $j\in\{0,1\}$. 

We say that $T$ is an \textit{embedding/quotient} of compatible couples if $T:E_j\to F_j$ is an embedding/quotient of normed spaces for $j\in\{0,1\}$ (recall that a bounded operator $T:E\to F$ between normed spaces is an embedding/quotient if it is injective/surjective and the induced bounded operator $E/\ker T\to\ima T$ is an isomorphism of normed spaces). We say that $T$ is an \textit{isomorphism} of compatible couples if $T:E_j\to F_j$ is an isomorphism of normed spaces, for $j\in\{0,1\}$. 

We say that $T$ is \textit{contractive} if $\|T\|_{(E_0,E_1)\to(E_0,E_1)}\pp1$. We say that $T$ is an \textit{isometric embedding/coisometric quotient} of compatible couples if $T:E_j\to F_j$ is an isometric embedding/coisometric quotient of normed spaces for $j\in\{0,1\}$ (recall that a quotient of normed spaces $T:E\to F$ is coisometric if the induced isomorphism of normed spaces $E/\ker T\to F$ is isometric). We say that $T$ is an \textit{isometric isomorphism} of compatible couples if $T:E_j\to F_j$ is an isometric isomorphism of normed spaces, for $j\in\{0,1\}$. 

\begin{rem}
There is an obvious way to define the category of compatible couples and compatible (contractive) bounded operators. The isomorphisms in this category correspond to the (isometric) isomorphisms of compatible couples.
\end{rem}

An \textit{interpolation space} with constant $C$ for a compatible couple $(E_0,E_1)$ is an intermediate space $E_\theta$ for $(E_0,E_1)$, such that, if $T:(E_0,E_1)\to(E_0,E_1)$ is a compatible bounded operator, then $T$ maps $E_\theta$ into $E_\theta$ and $\|T\|_{E_\theta\to E_\theta}\pp C\|T\|_{(E_0,E_1)\to(E_0,E_1)}$. An \textit{exact interpolation space} is an interpolation space with constant $C\pp1$. The sum/intersection of (exact) interpolation spaces is again an (exact) interpolation space. More generally, an \textit{interpolation pair} with constant $C$ for a pair of compatible couples $(E_0,E_1)$ and $(F_0,F_1)$ is a pair of intermediate spaces $E_\theta$ and $F_\theta$ for $(E_0,E_1)$ and $(F_0,F_1)$ respectively, such that, if $T:(E_0,E_1)\to(F_0,F_1)$ is a compatible bounded operator, then $T$ maps $E_\theta$ into $F_\theta$ and $\|T\|_{E_\theta\to F_\theta}\pp C\|T\|_{(E_0,E_1)\to(F_0,F_1)}$.

An \textit{exact interpolation pair} is an interpolation pair with constant $C\pp1$. 

\subsection{Interpolation functors}

An \textit{interpolation functor} with constant $C$ is a map $\mathcal{F}$ that assigns to each compatible couple $(E_0,E_1)$ an intermediate space $\mathcal{F}(E_0,E_1)$, such that, if $(E_0,E_1)$ and $(F_0,F_1)$ are a pair of compatible couples, then $\mathcal{F}(E_0,E_1)$ and $\mathcal{F}(F_0,F_1)$ is an interpolation pair with constant $C$ for $(E_0,E_1)$ and $(F_0,F_1)$ (in this situation, if $(E_0,E_1)$ is a compatible couple, then $\mathcal{F}(E_0,E_1)$ is necessarily an interpolation space with constant $C$ for $(E_0,E_1)$).

An \textit{exact interpolation functor} is an interpolation functor with constant $C\pp1$. 

\begin{rem}
For instance, the map $\Sigma$ (resp. $\Delta$) that assigns to each compatible couple $(E_0,E_1)$ the sum space $E_0+E_1$ (resp. the intersection space $E_0\cap E_1$) is an exact interpolation functor.
\end{rem}

If $\mathcal{F}$ is an (exact) interpolation functor, then $\mathcal{F}$ defines in an obvious way a functor from the category of compatible couples and compatible (contractive) bounded operators to the category of complete normed spaces and (contractive) bounded operators. 

The following result is the so-called Aronszajn-Gagliardo theorem.

\begin{theo}[\cite{BerghInterpolation}(Theorem 2.5.1)]
Let $E_\theta$ be an exact interpolation space for a compatible couple $(E_0,E_1)$. Then there is an exact interpolation functor $\mathcal{F}$ such that $E_\theta=\mathcal{F}(E_0,E_1)$.
\end{theo}

\subsection{Subcouples} 

A \textit{subcouple} of a compatible couple $(E_0,E_1)$ is a couple $(A_0,A_1)$ where $A_j$ is a closed subspace of $E_j$ for $j\in\{0,1\}$. In this situation, the couple $(A_0,A_1)$ inherits a unique structure of compatible couple making the inclusion $A_0+A_1\to E_0+E_1$ an isometric embedding of compatible couples $(A_0,A_1)\to(E_0,E_1)$. Thus, if $\mathcal{F}$ is an (exact) interpolation functor, then the inclusion
\[\fonctbis{\mathcal{F}(A_0,A_1)}{\mathcal{F}(E_0,E_1)}{u}{u}\]
is a well-defined bounded (contractive) injective operator, but it may fail to be an embedding of normed spaces (this is in particular the case for $\mathcal{F}=\Sigma$, but note that, however, in the case $\mathcal{F}=\Delta$, the inclusion $A_0\cap A_1\to E_0\cap E_1$ is an isometric embedding of normed spaces).

\subsection{Complementation} 

A subcouple $(A_0,A_1)$ of a compatible couple $(E_0,E_1)$ is \textit{complemented} with constant $C$ if there is a bounded (contractive) compatible operator $P:(E_0,E_1)\to(E_0,E_1)$ such that $P:E_j\to E_j$ is idempotent with range $A_j$, and $\|P\|_{E_j\to E_j}\pp C$, for $j\in\{0,1\}$.

A subcouple $(A_0,A_1)$ of a compatible couple $(E_0,E_1)$ is \textit{contractively complemented} if it is complemented with constant $C\pp1$.

If $(A_0,A_1)$ is a (contractively) complemented subcouple of a compatible couple $(E_0,E_1)$, then for every (exact) interpolation functor $\mathcal{F}$, the inclusion $\mathcal{F}(A_0,A_1)\to\mathcal{F}(E_0,E_1)$ is an (isometric) embedding of normed spaces, with range $(A_0+A_1)\cap\mathcal{F}(E_0,E_1)$, and the projection
$\mathcal{F}(E_0,E_1)\to\mathcal{F}(E_0/A_0,E_1/A_1)$ is a (coisometric) quotient of normed spaces, with kernel $\mathcal{F}(A_0,A_1)$.

\subsection{\texorpdfstring{$K$}{}-functionals} 

Let $(E_0,E_1)$ be a compatible couple. The \textit{$K$-functional} of $u\in E_0+E_1$ is defined for $t>0$ as
\[K_t(u)=K_t(u,E_0,E_1):=\inf\big\{\|u_0\|_{E_0}+t\|u_1\|_{E_1}\ :\ u_0\in E_0,\ u_1\in E_1,\ u=u_0+u_1\big\}.\]
For fixed $t>0$, $K_t$ is an equivalent norm on $E_0+E_1$. If $(E_0,E_1)$ and $(F_0,F_1)$ are two compatible couples and $T:(E_0,E_1)\to(F_0,F_1)$ a compatible bounded operator, then \[K_t(Tu,F_0,F_1)\pp\|T\|_{(E_0,E_1)\to(F_0,F_1)}K_t(u,E_0,E_1)\]
for every $u\in E_0+E_1$ and $t>0$. 

A \textit{$K$\-/method parameter} is a complete normed space $\Phi$ of (class of) Lebesgue measurable functions with variable $t>0$ such that,
\begin{enumerate}[nosep]
    \item[$\triangleright$] if $f,g\in\Phi$ with $|g|\pp|f|$ then $\|g\|_{\Phi}\pp\|f\|_{\Phi}$, 
    \item[$\triangleright$] the function $t\mapsto1\wedge t$ belongs to $\Phi$.
\end{enumerate}
If $\Phi$ is a $K$\-/method parameter and $(E_0,E_1)$ is a compatible couple, then the \textit{$K$-method interpolation space}
\[K_\Phi(E_0,E_1):=\big\{u\in E_0+E_1\ :\ t\mapsto K_t(u,E_0,E_1)\in\Phi\big\}\]
is a subspace of $E_0+E_1$ and is equipped with the complete norm $\|\cdot\|_{K_\Phi(E_0,E_1)}$ given by the expression
\[\|u\|_{K_\Phi(E_0,E_1)}:=\|t\mapsto K_t(u,E_0,E_1)\|_{\Phi}.\]
This construction defines an exact interpolation functor $K_\Phi$ called the \textit{$K$\-/method} with parameter $\Phi$. 

If $(A_0,A_1)$ is a subcouple of a compatible couple $(E_0,E_1)$, then we have
\[K_t(u,E_0,E_1)\pp K_t(u,A_0,A_1),\ \ \ \ \text{for}\ u\in A_0+A_1,\ t>0\]
A subcouple $(A_0,A_1)$ of a compatible couple $(E_0,E_1)$ is \textit{$K$\-/closed} with constant $C$ if 
\[K_t(u,A_0,A_1)\pp CK_t(u,E_0,E_1)\]
for every $u\in A_0+A_1$ and $t>0$. 

\begin{rem}
If $(A_0,A_1)$ is a $K$-closed subcouple of a compatible couple $(E_0,E_1)$, then $A_0+A_1$ is closed in $E_0+E_1$ (in other words, the inclusion $A_0+A_1\to E_0+E_1$ is an embedding of normed spaces).
\end{rem}

The following proposition follows directly from the definitions.

\begin{prop}\label{preliminaries_Kclosedness}
If $(A_0,A_1)$ is a $K$-closed subcouple of a compatible couple $(E_0,E_1)$, then for every $K$\-/method parameter $\Phi$, the inclusion $K_\Phi(A_0,A_1)\to K_\Phi(E_0,E_1)$ is an embedding of normed spaces, with range $(A_0+A_1)\cap K_\Phi(E_0,E_1)$.
\end{prop}

The notion of $K$-closedness used in \cite{KislyakovXu} is slightly different from the one we use here. In order to avoid any confusion we introduce a new terminology.

\begin{defi}
A subcouple $(A_0,A_1)$ of a compatible couple $(E_0,E_1)$ is \textit{quasi\-/complemented} with constant $C$ if, for every $a\in A_0+A_1$, if $a=u_0+u_1$ with $u_j\in E_j$ for $j\in\{0,1\}$, then $a=a_0+a_1$ with $a_j\in A_j$ and $\|a_j\|_{E_j}\pp C\|u_j\|_{E_j}$ for $j\in\{0,1\}$.
\end{defi}

Let $(A_0,A_1)$ be a subcouple of a compatible couple $(E_0,E_1)$. If $(A_0,A_1)$ is complemented in $(E_0,E_1)$ with constant $C$, then $(A_0,A_1)$ is clearly quasi-complemented in $(E_0,E_1)$ with constant $C$. Moreover, it is clear that if $(A_0,A_1)$ is quasi-complemented in $(E_0,E_1)$ with constant $C$, then it is $K$-closed in $(E_0,E_1)$ with constant $C$, and in addition we have $A_j=E_j\cap(E_0+E_1)$ for $j\in\{0,1\}$, so that it is also normal (and thus $J$-closed) in $(E_0,E_1)$. 

\begin{prop}
Let $(A_0,A_1)$ be a subcouple of a compatible couple $(E_0,E_1)$ such that $A_j=E_j\cap(E_0+E_1)$ for $j\in\{0,1\}$ and which is $K$-closed in $(E_0,E_1)$ with constant $C$. Then $(A_0,A_1)$ is quasi-complemented in $(E_0,E_1)$ with constant $4C$.
\end{prop}
\begin{proof}
Let $a\in A_0+A_1$, and write $a=u_0+u_1$ with $u_j\in E_j$ for $j\in\{0,1\}$. If $u_j=0$, then $a\in(A_0+A_1)\cap E_j=A_j$ and there is nothing to prove. Thus, we can assume $u_0,u_1\neq0$. Set $t:=\|u_0\|_{E_0}/\|u_1\|_{E_1}$. Then, we can write $a=a_0+a_1$ with $a_j\in A_j$ and 
\[\|a_0\|_{E_0}+t\|a_1\|_{E_1}\pp 2K_t(a,A_0,A_1)\pp 2C K_t(a,E_0,E_1)\pp 2C(\|u_0\|_{E_0}+t\|u_1\|_{E_1})=4C\|u_0\|_{E_0}.\]
This gives the estimates
\[\|a_0\|_{E_0}\pp 4C\|u_0\|_{E_0},\ \ \ \ \|a_1\|_{E_1}\pp 4C\|u_1\|_{E_1}.\]
\end{proof}

\section{Real and Complex interpolation spaces}

\subsection{The real method}

Let $0<\theta<1$ and $1\pp p\pp\infty$. Let $\Phi_{\theta,p}$ denote the space of Lebesgue\-/measurable functions $f$ with variable $t>0$ such that 
\[\|f\|_{\Phi_{\theta,p}}:=\|t\mapsto t^{-\theta}f(t)\|_{L_p(dt/t)}<\infty\]
Then $\Phi_{\theta,p}$ is a $J$\-/method parameter (and thus also a $K$\-/method parameter). If $(E_0,E_1)$ is a compatible couple, let us denote \[(E_0,E_1)_{\theta,p,K}:=K_{\Phi_{\theta,p}}.\]
By convention, we set $(E_0,E_1)_{0,p,K}:=E_0$ and $(E_0,E_1)_{1,p,K}:=E_1$ for every $1\pp p\pp\infty$. They are the so called \textit{real interpolation spaces}.

\subsection{The complex method}

Let $\S:=\{z\in\C\ :\ 0<\re z<1\}$ denote the open unit strip in the complex plane, with boundary $\partial\S=\{it\ :\ t\in\R\}\cup\{1+it\ :\ t\in\R\}$ and closure $\overline{\S}:=\{z\in\C\ :\ 0\pp\re z\pp1\}$. 

Let $(E_0,E_1)$ be a compatible couple. Let $\mathcal{F}(E_0,E_1)$ denote the space of functions $f:\overline{\S}\to E_0+E_1$ continuous on $\overline{\S}$, holomorphic on $\S$, such that $f(j+it)\in E_j$ for $t\in\R$, $j\in\{0,1\}$, and such that the functions $\R\to E_j$, $t\mapsto f(j+it)$ are continuous and vanishes at infinity, for $j\in\{0,1\}$. If $f\in\mathcal{F}(E_0,F_1)$, we set
\[\|f\|_{\mathcal{F}(E_0,E_1)}:=\max_{j\in\{0,1\}}\sup_{t\in\R}\|f(j+it)\|_{E_j}.\]
If $0<\theta<1$, the \textit{complex interpolation space} $[E_0,E_1]_\theta$ is the subspace of $E_0+E_1$ of elements of the form $f(\theta)$ with $f\in\mathcal{F}(E_0,E_1)$. It is equipped with the complete norm $\|\cdot\|_{[E_0,E_1]_\theta}$ given by the expression
\[\|u\|_{[E_0,E_1]_\theta}:=\inf\big\{\|f\|_{\mathcal{F}(E_0,E_1)}\ :\ f\in\mathcal{F}(E_0,E_1),\ u=f(\theta)\big\}.\]
This construction yields, for fixed $0<\theta<1$, an exact interpolation functor. By convention, we set $(E_0,E_1)_{0}:=E_0$ and $(E_0,E_1)_{1}:=E_1$.

\section{Noncommutative integration theory}

The material in this section is taken from \cite{DoddsSukochevIntegration}. An other relevant reference is \cite{DoddsInterpolationLp}.

\subsection{Semifinite von Neumann algebras}

Let $M$ be a \textit{(semi)finite von Neumann algebra}, i.e. a von Neumann algebra equipped with a normal (semi)finite faithful (n.s.f.) trace $\tau$. Let $H$ denote the Hilbert space on which $M$ acts. A closed and densely defined operator $x$ on $H$ with polar decomposition $x=u|x|$ and spectral decomposition $|x|=\int_{0}^{\infty}sde_s$ is \textit{affiliated} with $M$ if $u\in M$ and $e_s\in M$ for all $s>0$. The \textit{distribution function} of $x$ is the right-continuous nonnegative decreasing function of the variable $s>0$ denoted $\lambda_x$, such that
\[\lambda_x(s)=\tau(1-e_s),\ \ \ \text{for}\ s>0.\]
The \textit{singular function} of $x$ is the right-continuous nonnegative decreasing function of the variable $s>0$ denoted $\mu_x$ such that
\[\mu_x(s)=\inf\big\{t>0\ :\ \lambda_x(t)\pp s\big\},\ \ \text{for}\ s>0.\]
A closed and densely defined operator $x$ on $H$ is \textit{$\tau$\-/measurable} if it is affiliated with $M$ and if its distribution function (or its singular function) takes at least one finite value. Any element of $M$ is $\tau$\-/measurable. The set $L_0(M)$ of $\tau$\-/measurable operators then admits a canonical structure of $\ast$\-/algebra, so that the inclusion $M\to L_0(M)$ is a $\ast$-morphism and $\tau$ is canonically extended to the positive part $L_0(M)_+$ of $L_0(M)$ so that
\[\tau(x)=\int_{0}^{\infty}\mu_x(s)ds=\int_{0}^{\infty}\lambda_x(s)ds,\ \ \ \text{for}\ x\in L_0(M)_+.\]
For every $x\in L_0(M)$ and $1\pp p\pp\infty$, we set 
\[\|x\|_{L_p(M)}=\|x\|_p:=\left\{\begin{array}{cl}
    \Big[\displaystyle\int_{0}^{\infty}\mu_x(s)^pds\Big]^{1/p}=\Big[\int_{0}^{\infty}\lambda_x(s)ps^{p-1}ds\Big]^{1/p} & \text{if}\ p<\infty \\
    \inf\{s>0\ \mid\ \tau(1_{(s,\infty)}(|x|))=0\} & \text{if}\ p=\infty
\end{array}\right..\]
Then, for every $1\pp p\pp\infty$, the associated \textit{$L_p$-space},
\[L_p(M):=\big\{x\in L_0(M)\ :\ \|x\|_{L_p(M)}<\infty\big\}\]
is a subspace of $L_0(M)$ and becomes a complete normed space under $\|\cdot\|_{L_p(M)}$. Note that we have $L_\infty(M)=M$ with equal norms.

We will need the following version of Marcinkiewicz interpolation theorem.

\begin{theo}[Marcinkiewicz]
Let $T:L_2(M)\to L_2(M)$ be a self-adjoint bounded operator which is \textit{of weak-type} $(1,1)$ with constant $C$, i.e. for every $x\in (L_1\cap L_2)(M)$, we have
\[\sup_{s>0}s\lambda_{T(x)}(s)\pp C\|x\|_{L_1(M)}.\]
Then, for every $1<p<\infty$, the operator $T$ is $L_p(M)$-bounded with a constant depending only on $p,\|T\|_{L_2(M)\to L_2(M)}$ and $C$.
\end{theo}

\subsection{The compatible couple \texorpdfstring{$(L_1,L_\infty)$}{}}

The $\ast$-algebra $L_0(M)$ admits a canonical Hausdorff topology so that the inclusions $L_1(M)\to L_0(M)$ and $L_\infty(M)\to L_0(M)$ become continuous, making the couple $(L_1(M),L_\infty(M))$ compatible. 

\begin{prop}
Let $E(M)$ be an exact interpolation space for $(L_1(M),L_\infty(M))$. 
\begin{enumerate}[nosep]
    \item If $x\in E(M)$, then $x^*\in E(M)$, with $\|x^*\|_{E(M)}=\|x\|_{E(M)}$.
    \item If $x\in E(M)$, $a,b\in M$, then $axb\in E(M)$ with $\|axb\|_{E(M)}\pp\|a\|_M\|x\|_{E(M)}\|b\|_M$.
\end{enumerate}
\end{prop}

The following results are classical.

\begin{theo}
Let $1<p<\infty$. Then 
\[L_p(M)=[L_1(M),L_\infty(M)]_{1/q}\]
with equal norms, where $1<q<\infty$ is such that $p^{-1}+q^{-1}=1$.
\end{theo}

\begin{theo}[\cite{DoddsSukochevIntegration}(Theorem 3.9.16)]\label{integration_Holmstedt}
Let $x\in L_0(M)$. Then $x\in L_1(M)+L_\infty(M)$ if and only if for every $t>0$, we have
\[\int_{0}^{t}\mu_x(s)ds<\infty\]
and in that case, we have
\[K_t(x,L_1(M),L_\infty(M))=\int_{0}^{t}\mu_x(s)ds,\ \ \ \ \text{for}\ t>0.\]
\end{theo}

\begin{coro}
Let $1<p<\infty$. Then 
\[L_p(M)=(L_1(M),L_\infty(M))_{1/q,p,K}\]
with equivalent norms, with constants depending on $p$ only where $1<q<\infty$ is such that $p^{-1}+q^{-1}=1$.
\end{coro}

The following deep result says that the compatible couple $(L_1(M),L_\infty(M))$ is a Calder\'on couple. We refer the reader to \cite{KaltonInterpolation} for more background on Calder\'on couples.

\begin{theo}[\cite{DoddsSukochevIntegration}(Corollary 7.2.3)]\label{Integration_Calderon}
Let $E(M)$ be an intermediate space for $(L_1(M),L_\infty(M))$. Then $E(M)$ is an exact interpolation space for $(L_1(M),L_\infty(M))$ if and only if it is fully symmetric, i.e. for every $x,y\in L_0(M)$ such that $K_t(x)\pp K_t(y)$ for all $t>0$, if we have $y\in E(M)$, then $x\in E(M)$ with $\|x\|_{E(M)}\pp\|y\|_{E(M)}$. 
\end{theo}

\begin{rem}
As a consequence, if $E(M)$ is an exact interpolation space for $(L_1(M),L_\infty(M))$, and if $x\in E(M)$, then $|x|\in E(M)$ with $\||x|\|_{E(M)}=\|x\|_{E(M)}$.
\end{rem}

We consider the commutative semifinite von Neumann algebra $L_\infty(\R_+)$ of bounded Lebesgue\-/measurable functions on the half-line $\R_+$ whose trace is given by the Lebesgue measure. For $1\pp p\pp\infty$, we will denote $L_p:=L_p(L_\infty(\R_+))$. Note that, if $x\in(L_1+L_\infty)(M)$, then its singular function $\mu_x$ lives in $L_1+L_\infty$, as an immediate consequence of \ref{integration_Holmstedt}. If $E$ is an exact interpolation space for $(L_1,L_\infty)$, then
\[E(M):=\big\{x\in(L_1+L_\infty)(M)\ :\ \mu_x\in E \big\}\]
is a subspace of $(L_1+L_\infty)(M)$ and is equipped with the complete norm $\|\cdot\|_{E(M)}$ given by the expression
\[\|x\|_{E(M)}:=\|\mu_x\|_E,\ \ \ \text{for}\ x\in E(M).\]

\begin{theo}[\cite{DoddsSukochevIntegration}(Corollary 7.2.4)]
If $E$ is an exact interpolation space for $(L_1,L_\infty)$, then $E(M)$ is an exact interpolation space for $(L_1(M),L_\infty(M))$.
\end{theo}

\begin{theo}[\cite{DoddsSukochevIntegration}(Theorem 7.3.4)]
Let $E_0,E_1$ be exact interpolation spaces for $(L_1,L_\infty)$, let $\mathcal{F}$ be an exact interpolation functor, and set $E_\theta:=\mathcal{F}(E_0,E_1)$. Then
\[E_\theta(M)=\mathcal{F}(E_0(M),E_1(M))\]
with equal norms.
\end{theo}

\begin{rem}
As a consequence of the above result together with the Aronszajn–Gagliardo theorem, if $E_0,E_1$ are exact interpolation spaces for $(L_1,L_\infty)$, then the exact interpolation spaces for $(E_0(M),E_1(M))$ are those of the form $E_\theta(M)$ where $E_\theta$ is an exact interpalation space for $(E_0,E_1)$.
\end{rem}

\subsection{Trace duality}

Let $M$ be a semifinite von Neumann algebra, with trace denoted $\tau$. Then the trace $\tau$ extends to a positive and contractive linear form on $L_1(M)$, again denoted $\tau$.

If $E(M)$ is an exact interpolation space for $(L_1(M),L_\infty(M))$, its \textit{K\"othe dual} 
\[E^{\times}(M):=\big\{y\in L_0(M)\ :\ \forall x\in E(M),\ xy\in L_1(M)\big\}\]
\[=\big\{y\in L_0(M)\ :\ \forall x\in E(M),\ yx\in L_1(M)\big\}\]
is a subspace of $L_0(M)$ and is equipped with the complete norm $\|\cdot\|_{E^{\times}(M)}$ given by the expression
\[\|y\|_{E^{\times}(M)}=\sup_{x\in E(M),\ \|x\|_{E(M)}\pp1}|\tau(xy)|=\sup_{x\in E(M),\ \|x\|_{E(M)}\pp1}|\tau(yx)|,\ \ \ \ \ \ \text{for}\ y\in E^{\times}(M).\]
 By definition, if $x\in E(M)$ and $y\in E^{\times}(M)$, then $xy,yx\in L_1(M)$ and 
\[\max\{\|xy\|_{L_1(M)},\|yx\|_{L_1(M)}\}\pp\|x\|_{E(M)}\|y\|_{E^{\times}(M)}\ \ \ \ \ \ \ \ \ \textit{(generalized H\"older's inequality)}\]
If $E$ is an exact interpolation space for $(L_1,L_\infty)$, and if $F:=E^{\times}$ denote its K\"othe dual, then by \cite{DoddsSukochevIntegration}[Theorem 6.1.6], the K\"othe dual of $E(M)$ is $F(M)$ (with equal norms). Hence, the notation $E^{\times}(M)$ is consistent with the notations introduced in the previous paragraph. The following results are consequence of \cite{DoddsSukochevIntegration}(Proposition 3.4.8, Theorem 3.4.24, Corollary 4.3.9, Corollary 6.2.1).

\begin{theo}
Let $1\pp p,q\pp\infty$ such that $p^{-1}+q^{-1}=1$. Then, we have $L_p(M)^{\times}=L_q(M)$ with equal norms. Moreover, we have $(L_1\cap L_\infty)(M)^{\times}=(L_1+L_\infty)(M)$ and $(L_1+L_\infty)(M)^{\times}=(L_1\cap L_\infty)(M)$ with equal norms.
\end{theo}

\begin{theo}
Let $E(M)$ be an exact interpolation space for $(L_1(M),L_\infty(M))$. Then the K\"othe dual $E^{\times}(M)$ is an exact interpolation space for $(L_1(M),L_\infty(M))$. Moreover, the sesquilinear form
\[\fonctbis{E(M)\times E^{\times}(M)}{\C}{(x,y)}{\tau(xy^*)}\]
is non-degenerate, and thus, defines a canonical duality between $E(M)$ and $E^{\times}(M)$, called \textbf{trace duality}. Finally, $(L_1\cap L_\infty)(M)$ separates the points of both $E(M)$ and $E^{\times}(M)$ w.r.t. this duality.
\end{theo}


Let $E(M)$ be an exact interpolation space for $(L_1(M),L_\infty(M))$, and let $E^*(M)$ denote the dual space of $E(M)$ as a normed space. Then trace duality yields a canonical isometric (antilinear) operator $E^{\times}(M)\to E^*(M)$, but in general it may not be surjective. By \cite{DoddsSukochevIntegration}[Theorem 5.2.9, Proposition 5.3.2], it is surjective if and only if the norm of $E(M)$ is \textit{order\-/continuous}, i.e. if for every decreasing net $(x_\alpha)_\alpha$ of $E(M)_+$ such that $\inf_\alpha x_\alpha=0$ then $\inf_\alpha\|x_\alpha\|_{E(M)}=0$. Thus, if $E(M)$ is an exact interpolation space for $(L_1(M),L_\infty(M))$ with order\-/continuous norm, then $E(M)$ is a predual of $E^{\times}(M)$ so that $E^{\times}(M)$ inherits a canonical w*-topology. For instance, if $1\pp p<\infty$, then $L_p(M)$ has an order\-/continuous norm.

The following result directly follows.

\begin{theo}
Let $E(M)$ be an exact interpolation space for $(L_1(M),L_\infty(M))$ with order-continuous norm. Then the subspace $(L_1\cap L_\infty)(M)$ is norm-dense in $E(M)$ and \text{\upshape w*}-dense in $E^{\times}(M)$.
\end{theo}

\section{Interpolation of abstract Hardy-type spaces}

Let $M$ be a semifinite von Neumann algebra. 

\vspace{5pt}

\textbf{Notations.} If $P$ is any (self-adjoint) projection on $L_2(M)$ and if $E(M)$ is any exact interpolation space for $(L_1(M),L_\infty(M))$, we define $H_E(P)$ to be the $\sigma(E(M),E^{\times}(M))$\-/weak closure of the subspace
\[E(M)\cap P(L_2(M))\]
in $E(M)$. Let us denote $H_p(P):=H_{L_p}(P)$ for $1\pp p\pp\infty$. By definition, $H_2(P)=P(L_2(M))$. They are the \textit{abstract Hardy spaces} associated with $P$.

\begin{rem}
If $P$ is a projection on $L_2(M)$ with complement projection $P^{\bot}$, and if $E(M)$ is an exact interpolation space for $(L_1(M),L_\infty(M))$, then from the definitions it is clear that $H_E(P)$ and $H_{E^{\times}}(P)$ are respectively included in the orthogonal of $H_{E^{\times}}(P^\bot)$ and $H_E(P^{\bot})$ w.r.t. trace duality between $E(M)$ and $E^{\times}(M)$.
\end{rem}

\begin{mdframed}[backgroundcolor=black!10,rightline=false,leftline=false,topline=false,bottomline=false,skipabove=10pt]
A projection $P$ on $L_2(M)$ is said to satisfy the \textit{technical assumptions} if
\vspace{3pt}
\begin{enumerate}[nosep]
    \item[(1)] If $E(M)$ is an exact interpolation space for $(L_1(M),L_\infty(M))$ with order\-/continuous norm, then the orthogonals of $H_E(P)$, $H_{E^{\times}}(P)$ w.r.t. trace duality coincide respectively with $H_{E^{\times}}(P^{\bot})$, $H_E(P^{\bot})$.
    \item[(2)] There is a subspace $H(P)$ of $(L_1+L_\infty)(M)$, which is weakly-closed in $(L_1+L_\infty)(M)$ (w.r.t. trace duality between $(L_1+L_\infty)(M)$ and $(L_1\cap L_\infty)(M)$), and such that, if $E(M)$ is an exact interpolation space for $(L_1(M),L_\infty(M))$ with order\-/continuous norm, then we have 
\[H_E(P)=H(P)\cap E(M),\ \ \ \ H_{E^{\times}}(P)=H(P)\cap E^{\times}(M).\]
    \item[(3)] If $E(M)$ is an exact interpolation space for $(L_1(M),L_\infty(M))$ with order\-/continuous norm, then the subspace
\[(L_1\cap L_\infty)(M)\cap P(L_2(M))\]
is norm-dense in $H_E(P)$ and w*-dense in $H_{E^{\times}}(P)$.
\end{enumerate}
\end{mdframed}
\vspace{5pt}

\begin{prop}
Let $P$ be a projection on $L_2(M)$ which is $L_1(M)$-bounded. Then $P$ satisfies the technical assumptions.
\end{prop}
\begin{proof}
As $P$ is self-adjoint on $L_2(M)$, we see that $P$ is also $L_\infty(M)$-bounded, so that it extends to a bounded compatible operator $P:(L_1(M),L_\infty(M))\to(L_1(M),L_\infty(M))$. Then, it is easy to show that, if $E(M)$ is an exact interpolation space for $(L_1(M),L_\infty(M))$ with order\-/continuous norm, we have
\[H_E(M)=\big\{x\in E(M)\ :\ Px=x\big\},\ \ H_{E^{\times}}(M)=\big\{x\in E(M)\ :\ Px=x\big\}\]
By using the fact that $P$ is self-adjoint on $L_2(M)$, it is then easy to check that $P$ satisfies the technical assumptions.
\end{proof}

\begin{prop}
Let $P$ be a projection on $L_2(M)$ which satisfies the technical assumptions. Then the complement projection $P^{\bot}$ on $L_2(M)$ also satisfies the technical assumptions.
\end{prop}
\begin{proof}
As $P$ satisfies (1), by the bipolar theorem it is clear that $P^{\bot}$ satisfies (1). 

For (2), let $H(P^{\bot})$ denote the orthogonal of $(L_1\cap L_\infty)(M)\cap P(L_2(M))$ w.r.t. trace duality between $(L_1+L_\infty)(M)$ and $(L_1\cap L_\infty)(M)$. Then $H(P^{\bot})$ is indeed weakly-closed in $(L_1+L_\infty)(M)$. Let $E(M)$ be an exact interpolation space for $(L_1(M),L_\infty(M))$ with order\-/continuous norm. As $H_E(P^{\bot})$ is orthogonal to $H_{E^{\times}}(P)$, and because $H_{E^{\times}}(P)$ contains $(L_1\cap L_\infty)(M)\cap P(L_2(M))$, we deduce that $H_E(P^{\bot})$ is included in $H(P^{\bot})$, and thus in $E(M)\cap H(P^{\bot})$. For the converse inclusion, let $x\in E(M)\cap H(P^{\bot})$. As $P$ satisfies (3), we deduce that $x$ is orthogonal to $H_{E^{\times}}(P)$. As $P$ satisfies (1), we deduce that $x\in H_E(P^{\bot})$. Thus $H_E(P^{\bot})=E(M)\cap H(P^{\bot})$ as desired. In a similar way, we prove that $H_{E^{\times}}(P^{\bot})=E^{\times}(M)\cap H(P^{\bot})$. 

It remains to justify (3). Note that $(L_1\cap L_\infty)(M)\cap P^{\bot}(L_2(M))$ is weakly closed in $(L_1\cap L_\infty)(M)$, as a consequence of the fact that $P$ is weakly-continuous for the weak topology of $L_2(M)$. Thus, as $P$ satisfies (1), by the bipolar theorem it suffices to show that, if $x\in E^{\times}(M)$ resp. $x\in E(M)$ is orthogonal to $(L_1\cap L_\infty)(M)\cap P^{\bot}(L_2(M))$, then $x\in H_{E^\times}(P)$ (resp. $x\in H_E(P)$). Let $x\in E^{\times}(M)$ (resp. $x\in E(M)$) orthogonal to $(L_1\cap L_\infty)(M)\cap P^{\bot}(L_2(M))$. Then $x$ is in the biorthogonal of $H(P)$ w.r.t. trace duality between $(L_1+L_\infty)(M)$ and $(L_1\cap L_\infty)(M)$. By the bipolar theorem, and because $H(P)$ is, by assumption, weakly closed in $(L_1+L_\infty)(M)$, we deduce that $x\in H(P)$. Thus $x\in E^{\times}(M)\cap H(P)$ (resp. $x\in E^{\times}(M)\cap H(P)$). As $P$ satisfies (2), we finally deduce that $x\in H_{E^{\times}}(P)$ (resp. $x\in H_E(P)$), as desired.
\end{proof}

\begin{prop}\label{technical_assumptions_Kclosedness}
Let $P$ be a projection on $L_2(M)$ which satisfies the technical assumptions, let $E_0(M)$, $E_1(M)$ be two exact interpolation spaces for $(L_1(M),L_\infty(M))$ such that the subcouple $(H_{E_0}(P),H_{E_1}(P))$ is quasi-complemented in $(E_0(M),E_1(M))$ with constant $C$, and let $\Phi$ be a $K$-parameter space such that the exact interpolation space
\[E_\theta(M):=K_\Phi(E_0(M),E_1(M))\]
has order-continuous norm. Then
\[H_{E_\theta}(P)=K_\Phi(E_0(M),E_1(M))\]
with equivalent norms, with constants depending on $C$ only.
\end{prop}
\begin{proof}
By Proposition \ref{preliminaries_Kclosedness}, it suffices to check that
\[(H_{E_0}(P)+H_{E_1}(P))\cap E_\theta(M)\]
is a dense subspace of $H_{E_\theta}(P)$. It is a subspace because of point (2) of the technical assumptions, and it is dense because of point (3) of the technical assumptions.
\end{proof}

\begin{mdframed}[backgroundcolor=black!10,rightline=false,leftline=false,topline=false,bottomline=false,skipabove=10pt]
A projection $P$ on $L_2(M)$ is said to be of \textit{Jones-type} if
\vspace{3pt}
\begin{enumerate}[nosep]
    \item[(1)] $P$ satisfies the technical assumptions.
    \item[(2)] If $1<p<\infty$, then $P$ is $L_p(M)$-bounded with a constant depending on $p$ only.
    \item[(3)] The subcouples $(H_1(P),H_2(P))$ and $(H_1(P^{\bot}),H_2(P^{\bot}))$ are both quasi\-/complemented in the compatible couple $(L_1(M),L_2(M))$ with a universal constant.
\end{enumerate}
\end{mdframed}
\vspace{5pt}

\begin{rem}
If $P$ is a projection $L_2(M)$ which is of Jones-type, then the complement projection $P^{\bot}$ is also of Jones-type.
\end{rem}

A detailed proof of the following result is given in \cite{Moyart2026}[Theorem 2.6]. It essentially relies on the Wolff-type theorem for $K$-functionals established by Kislyakov and Xu in \cite{KisliakovHardy}[Theorem 2].

\begin{mdframed}[skipabove=10pt]
\begin{theo}
Let $P$ be a Jones-type projection on $L_2(M)$. If $1\pp p_0,p_0\pp\infty$, then the subcouple $(H_{p_0}(P),H_{p_1}(P))$ is quasi\-/complemented in the compatible couple $(L_{p_0}(M),L_{p_1}(M))$ with a constant depending only on $p_0,p_1$.
\end{theo}
\end{mdframed}
\vspace{5pt}

By Proposition \ref{technical_assumptions_Kclosedness}, we deduce the following result.

\begin{theo}
Let $P$ be a Jones-type projection on $L_2(M)$. Let $1\pp p_0,p_1\pp\infty$ and let $\Phi$ be a $K$-parameter space such that the exact interpolation space
\[E(M):=K_\Phi(L_{p_0}(M),L_{p_1}(M))\]
has order-continuous norm. Then 
\[H_E(P)=K_\Phi(H_{p_0}(P),H_{p_1}(P))\]
with equivalent norms, with constants depending only on $p_0,p_1$. 
\end{theo}

\begin{coro}
Let $P$ be a Jones-type projection on $L_2(M)$, and let $1<p<\infty$. Then
\[H_p(P)=(H_1(P),H_{\infty}(P))_{\theta,p_\theta,K}\]
with equivalent norms, with constants depending on $p$ only, where $\theta:=1-1/p$.
\end{coro}

The next result is an extended version of Pisier's method for the computation of complex interpolation spaces of abstract Hardy spaces, first introduced in \cite{PisierHardy}. A detailed proof of the extended version is given in \cite{Moyart2026}[Theorem A].

\begin{mdframed}[backgroundcolor=black!10,rightline=false,leftline=false,topline=false,bottomline=false,skipabove=10pt]
A projection $P$ on $L_2(M)$ is said to satisfy \textit{Pisier's method assumptions} if, for every semifinite von Neumann algebra $N$, the amplified projection $Q:=P\otimes I$ on the Hilbertian tensor product $L_2(M)\otimes L_2(N)=L_2(M\bar{\otimes}N)$ is of Jones-type, and the algebraic tensor product $H_\infty(P^{\bot})\odot L_\infty(N)$ is a subspace of $H_\infty(Q^{\bot})$.
\end{mdframed}
\vspace{5pt}

\begin{mdframed}[skipabove=10pt]
\begin{maintheo}[Pisier's method]\label{theorem_PisierMethod}
Let $P$ be a projection on $L_2(M)$ that satisfies Pisier's method assumptions.  If $1\pp p_0,p_1<\infty$ and $0<\theta<1$ then
\[[H_{p_0}(P),H_{p_1}(P)]_\theta=H_{p_\theta}(P)\]
with equivalent norms, with constants depending on $p_0,p_1,\theta$ only, where $\frac{1}{p_\theta}=\frac{1-\theta}{p_0}+\frac{\theta}{p_1}$.
\end{maintheo}
\end{mdframed}
\vspace{5pt}

\clearpage

\part{Bourgain's method in the semicommutative setting}

\section{Operators of Calderón-Zygmund type} 

Let $M$ be a semifinite von Neumann algebra with trace denoted $\tau$.

\begin{defi}
A bounded operator $T$ on $L_2(M)$ is of \textit{Calderón-Zygmund type} with constant $C$ if for every $s>0$ and $y\in(L_1\cap L_2)(M)$ there is a decomposition 
\[y=a+b\]
with $a,b\in L_2(M)$ and a projection $p\in M$ such that 
\begin{itemize}[nosep]
    \item $\|a\|_2^2\pp C^2s\|y\|_1$,
    \item $\tau(1-p)\pp C^2s^{-1}\|y\|_1$,
    \item $\|pT(b)p\|_1\pp C\|y\|_1$.
\end{itemize}
\end{defi}

\begin{rem}
Every bounded operator $T$ on $L_2(M)$ which is also $L_1$-bounded is of Calderón-Zygmund type with constant $\|T\|_{L_1\to L_1}$ (in the above definition, take $a=0$, $p=1$, and $b=y$).
\end{rem}

\begin{mdframed}[skipabove=10pt]
\begin{theo}\label{CZ_type_theorem1}
Let $T:L_2(M)\to L_2(M)$ be a bounded operator which is of Calderón-Zygmund type with constant $C$. Then $T$ is of weak type $(1,1)$ with a constant depending on $C,\|T\|_{L_2\to L_2}$ only.
\end{theo}
\end{mdframed}
\begin{proof}
Let $y\in(L_1\cap L_2)(M)$ and $s>0$. Then there is a decomposition
\[y=a+b\]
with $a,b\in L_2(N)$ and a projection $p\in N$ such that 
\begin{itemize}[nosep]
    \item $\|a\|_2^2\pp C^2s\|y\|_1$,
    \item $\tau(1-p)\pp C^2s^{-1}\|y\|_1$,
    \item $\|pT(b)p\|_1\pp C\|y\|_1$.
\end{itemize}
where $C$ is a universal constant. As 
\[T(y)=T(a)+T(b)=T(a)+(1-p)T(b)+pT(b)(1-p)+pT(b)p,\]
we have
\[\lambda_{T(y)}(4s)\pp\lambda_{T(a)}(s)+\lambda_{(1-p)T(b)}(s)+\lambda_{pT(b)(1-p)}(s)+\lambda_{pT(b)p}(s).\]
We clearly have
\[\lambda_{(1-p)T(b)}(s)\pp\sigma(1-p)\pp Cs^{-1}\|y\|_1,\ \ \ \lambda_{pT(b)(1-p)}(s)\pp\sigma(1-p)\pp Cs^{-1}\|y\|_1.\]
By the Markov inequality, we have
\[\lambda_{T(a)}(s)\pp s^{-2}\|T(a)\|_2^2\pp s^{-2}\|T\|_{L_2\to L_2}^2\|a\|_2^2\pp s^{-1}\|T\|_{L_2\to L_2}^2C^2\|y\|_1\]
and
\[\lambda_{pT(b)p}(s)\pp s^{-1}\|pT(b)p\|_1\pp Cs^{-1}\|y\|_1.\]
By combining the previous estimates, we get
\[\lambda_{T(y)}(4s)\pp C's^{-1}\|y\|_1\]
where $C':=3C+\|T\|_{L_2\to L_2}^2C^2$. Thus
\[\|Ty\|_{1,\infty}=\sup_{s>0}s\lambda_{T(y)}(s)\pp 4C'\|y\|_1\]
as desired.
\end{proof} 

\begin{mdframed}[skipabove=10pt]
\begin{theo}\label{CZ_type_theorem2}
Let $P:L_2(M)\to L_2(M)$ be a projection which satisfies the technical assumptions and which is of Calderón-Zygmund type with constant $C$. Then, the subcouple $(H_1(P),H_2(P))$ is quasi-complemented in $(L_1(M),L_2(M))$ with a constant depending on $C$ only.
\end{theo}
\end{mdframed}
\begin{proof}
By the technical assumptions, we already have \[H_1(P)=L_1(M)\cap(H_1(P)+H_2(P)),\ \ \ \ H_2(P)=L_2(M)\cap(H_1(P)+H_2(P)).\]
Thus, it suffices to show that $(H_1(P),H_2(P))$ is $K$-closed in $(L_1(M),L_2(M))$ with a constant depending on $C$ only. Let $t>0$, and let $x\in L_1(M)\cap P(L_2(M))$ such that 
\[K_t(x,L_1(M),L_2(M))\pp1\]
There is $y\in L_1(M)$, $z\in L_2(M)$ such that $x=y+z$ and $\|y\|_1+t\|z\|_2\pp2$. Then $y\in(L_1\cap L_2)(M)$ and by applying the above definition with $y$ and the parameter $s=t^{-2}$, there is a decomposition 
\[y=a+b\]
with $a,b\in L_2(M)$ and a projection $p\in M$ such that 
\begin{itemize}[nosep]
    \item $\|a\|_2^2\pp C^2t^{-2}\|y\|_1\pp 2C^2t^{-2}$,
    \item $\tau(1-p)\pp Ct^2\|y\|_1\pp 2Ct^2$,
    \item $\|pP(b)p\|_1\pp C\|y\|_1\pp 2C$,
\end{itemize}
Then, we set
\[y':=P(b),\ \ \ \ \ z':=P(a+z).\]
As $x=P(x)=P(y)+P(z)$, it is clear that $x=y'+z'$. On the one hand, as $P$ is $L_2$-contractive, we have
\[\|z'\|_2\pp\|P(a)\|_2+\|P(z)\|_2\pp\|a\|_2+\|z\|_2\pp \sqrt{2}Ct^{-1}+2t^{-1}\pp 2(C+1)t^{-1}\]
On the other hand, we can write $y'=py'p+(1-p)y'+py'(1-p)=pP(b)p+(1-p)y'+py'(1-p)$, so that
\begin{align*}
\|y'\|_1&\pp\|pP(b)p\|_1+\|(1-p)y'\|_1+\|py'(1-p)\|_1\\
&\pp 2C+ \|(1-p)y'\|_1+\|y'(1-p)\|_1.\\
\end{align*}
As $y'=x-z'=y+(z-z')$ we have
\begin{align*}
\|(1-p)y'\|_1&\pp\|(1-p)y\|_1+\|(1-p)(z-z')\|_1\\
&\pp2+\|1-p\|_2\|z-z'\|_2\\
&\pp2+\tau(1-p)^{1/2}(\|z\|_2+\|z'\|_2)\\
&\pp2+\tau(1-p)^{1/2}(2t^{-1}+2(C+1)t^{-1})\\
&\pp2+\sqrt{2}C^{1/2}t(2t^{-1}+(C+1)t^{-1})\\
&=2+2\sqrt{2}C^{1/2}(C+3)=:C',\\
\end{align*}
and similarly, we have
\[\|y'(1-p)\|_1\pp C'.\]
Thus
\[\|y'\|_1\pp 2C+2C'.\]
As we have $y'\in L_1(M)\cap P(L_2(M))$ and $z'\in P(L_2(M))$, this shows that
\[K_t(x,H_1(P),H_2(P))\pp\|y'\|_1+t\|z'\|_2\pp(2C+2C')+2(C+1)=:C''.\] 
Hence, we have shown that for every $x\in L_1(M)\cap P(L_2(M))$, we have
\[K_t(x,H_1(P),H_2(P))\pp C''K_t(x,L_1(M),L_2(M)).\] 
By the technical assumptions, we know that $L_1(M)\cap P(L_2(M))$ is dense in both $H_1(P)$ and $H_2(P)$. Thus, it is dense in $H_1(P)+H_2(P)$. As the $K$-functional $K_t(-,H_1(P),H_2(P))$ and $K_t(-,L_1(M),L_2(M))$ are both continuous on $H_1(P)+H_2(P)$, the above estimate extends to every $x\in H_1(P)+H_2(P)$. Hence, we obtain that $(H_1(P),H_2(P))$ is $K$-closed in $(L_1(M),L_2(M))$ with constant $C''$, as desired.
\end{proof}

%


\section{Singular kernel operators}

Let $\X$ be a \textit{Ahlfors $d$-regular locally compact space}, i.e. a locally compact space equipped with a Radon measure and a metric that induces its topology, for which the open balls of finite radius are relatively compact, and we have the Ahlfors $d$-regularity condition
\[|B_r(x)|\simeq r^d,\ \ \ \ \ \text{for}\ \ \ x\in\X,\ 0<r<\diam(\X).\]
Here $|\cdot|$ refers to the volume w.r.t. the measure and $B$ refers to the open balls w.r.t. the metric, and the notation $\simeq$ means that the estimate holds, up to some constants depending only on $\X$. We refer the reader to the appendix for more details.

\begin{rem}
Any $d$-Riemanian manifold equipped with its volume measure and its geodesic distance is an Ahlfors $d$-regular locally compact space \cite{ChristLectures}.
\end{rem}

In this situation, we have the following generalised doubling condition
\[|B_{\lambda r}(x)|\lesssim\lambda^d|B_r(x)|,\ \ \ \text{for}\ x\in\X,\ r>0,\ \lambda\pg1\]
and the following reverse doubling condition for small balls
\[|B_{\lambda r}(x)|\gtrsim\lambda^{d}|B_r(x)|,\ \ \ \text{for}\ x\in\X,\ r>0,\ \lambda\pp1.\]
In particular $\X$ is a space of homogeneous type in the usual sense of Coifman and Weiss \cite{CoifmanWeiss}. For $x,y\in\X$, we set
\[V(x,y):=|B_{d(x,y)}(x)|.\]
If $x,y\in\X$, we have (cf. appendix)
\[V(x,y)\lesssim V(y,x).\]
A \textit{kernel} on $\X$ is a scalar-valued continuous function defined on 
\[\big\{(x,y)\in \X\times\X\ :\ x\neq y\big\}.\]
Let $k$ be a kernel on $\X$. Then $k$ is said to be \textit{singular} if it satisfies the \textit{size condition}
\[|k(x,y)|\lesssim V(x,y)^{-1}\]
for $x,y\in\X$ with $x\neq y$, and the \textit{$L_2$-H\"ormander regularity condition}
\[\sum_{m=0}^{\infty}|B_{2^{m+1}r}(y')\setminus B_{2^mr}(y')|^{1/2}\sup_{y\in B_{r/3}(y')}\Big[\int_{B_{2^{m+1}r}(y')\setminus B_{2^mr}(y')}\big|k(x,y)-k(x,y')|^2dx\Big]^{1/2}\lesssim 1\]
for every $y'\in\X$ and $r>0$. This last condition is taken from \cite{CadilhacCZ2}. It is stronger than the usual H\"ormander regularity condition, but weaker than the classical Lipschitz regularity condition, as shown in the next proposition.

\begin{lemm}
If $y\in\X$, $m\pg0$ and $r,\alpha>0$ then
\begin{equation}
\Big[\int_{B_{2^{m+1}r}(y)\setminus B_{2^mr}(y)}\frac{d(x,y)^{-2\alpha}}{V(x,y)^2}dx\Big]^{1/2}\lesssim[2^mr]^{-\alpha}|B_{2^{m}r}(y)|^{-1/2}.
\end{equation}
\end{lemm}
\begin{proof}
\begin{align*}
    \Big[\int_{B_{2^{m+1}r}(y)\setminus B_{2^mr}(y)}\frac{d(x,y)^{-2\alpha}}{V(x,y)^2}dx\Big]^{1/2}&\lesssim[2^mr]^{-\alpha}\Big[\int_{B_{2^{m+1}r}(y)\setminus B_{2^{m}r}(y)}\frac{dx}{V(y,x)^2}\Big]^{1/2}\\
    &\pp[2^mr]^{-\alpha}\frac{1}{|B_{2^mr}(y)|}\Big[\int_{B_{2^{m+1}r}(y)\setminus B_{2^{m}r}(y)}dx\Big]^{1/2}\\
    &\pp[2^mr]^{-\alpha}\frac{|B_{2^{m+1}r}(y)|^{1/2}}{|B_{2^mr}(y)|}\\
    &\lesssim[2^mr]^{-\alpha}|B_{2^{m}r}(y)|^{-1/2}.
\end{align*}
\end{proof}

\begin{prop}
Let $k$ be a kernel on $\X$. If $k$ satisfies the Lipschitz regularity condition
\[|k(x,y)-k(x,y')|\lesssim V(x,y')^{-1}\frac{d(y,y')^\alpha}{d(x,y')^{\alpha}}\]
for $x,y,y'\in\X$ with $2d(y,y')<d(x,y)$, where $\alpha\in(0,1]$ is a constant, then $k$ satisfies the $L_2$-H\"ormander regularity condition.
\end{prop}
\begin{proof}
Let $y'\in\X$, $r>0$. If $y\in B_{r/3}(y')$ and $x\in\X\setminus B_{r}(y')$, then 
\[d(x,y)\pg d(x,y')-d(y,y')>r-r/3=2r/3\pg 2d(y,y')\]
thus, by the previous lemma we get
\begin{align*}
    \Big[\int_{B_{2^{m+1}r}(y')\setminus B_{2^mr}(y')}\big|k(x,y)-k(x,y')|^2dx\Big]^{1/2}&\lesssim\Big[\int_{B_{2^{m+1}r}(y')\setminus B_{2^mr}(y')}V(x,y')^{-2}\frac{d(y,y')^{2\alpha}}{d(x,y')^{2\alpha}}dx\Big]^{1/2}\\
    &\lesssim r^\alpha\Big[\int_{B_{2^{m+1}r}(y')\setminus B_{2^mr}(y')}\frac{d(x,y')^{-2\alpha}}{V(x,y')^2}dx\Big]^{1/2}\\
    &\lesssim 2^{-\alpha m}|B_{2^{m}r}(y')|^{-1/2}.
\end{align*}
As we have
\[\sum_{m=0}^{\infty}|B_{2^{m+1}r}(y')\setminus B_{2^mr}(y')|^{1/2}\times 2^{-\alpha m}|B_{2^mr}(y')|^{-1/2}\lesssim\sum_{m=0}^{\infty}2^{-\alpha m}\lesssim1\]
we get the desired conclusion.
\end{proof}

\vspace{10pt}

Let $M$ be a semifinite von Neumann algebra and let $N:=L_\infty(\X)\bar{\otimes}M$ be the tensor product von Neumann algebra equipped with the tensor product trace denoted $\sigma$. A bounded operator $T$ on $L_2(N)$ is \textit{given by a kernel} $k$ on $\X$ if $T=S\otimes I$, where $S$ is a bounded operator on $L_2(\X)$ that admits the integral representation
\[(Sf)(x)=\int_{\X}k(x,y)f(y)dy\]
for every $f\in L_2(\X)$ with compact support and $x\in\X$ outside the support of $f$. 

\begin{rem}
Note that if $T$ is a bounded operator on $L_2(N)$ given by a kernel on $\X$, then for every scalar $\lambda\in\C$, the bounded operator $T-\lambda I$ is also given by the same kernel on $\X$.
\end{rem}

The following theorem relies on the semicommutative Calderón-Zygmund decomposition introduced by J. Parcet, J.M. Conde Alonse and L. Cadilhac in \cite{CadilhacCZ2}. This construction is detailed in Appendix A.

\begin{mdframed}[skipabove=10pt]
\begin{theo}[Calderón-Zygmund decomposition]\label{singular_kernel_operators_CZ}
If $s>0$ and $f\in(L_1\cap L_2)(N)$ there is a decomposition
\[f=a+b\]
with $a,b\in L_2(N)$, and a projection $p\in N$, such that
\begin{itemize}[nosep]
    \item $\|a\|_2^2\pp C^2s\|y\|_1$,
    \item $\sigma(1-p)\pp C^2s^{-1}\|y\|_1$,
    \item $\|pT(b)p\|_1\pp C_T\|y\|_1$ for every bounded operator $T$ on $L_2(N)$ given by a singular kernel on $\X$, where $C_T$ is positive constant depending on $T$ only,
\end{itemize}
where $C$ is a constant depending on $\X$ only. 
\end{theo}
\end{mdframed}
\begin{proof}
Let $f\in(L_1\cap L_2)(N)$ and $s>0$. It is classicaly known that we have a decomposition
\[f=f_1-f_2+i(f_3-f_4)\]
with $f_j\in(L_1\cap L_2)(N)_+$ such that $\|f_j\|_1\pp\|f\|_1$ for $j\in\{1,\ldots,4\}$. For $j\in\{1,\ldots,4\}$, let
\[f_j=a_j+b_j\]
be the Calderón-Zygmund decompositions associated with $f_j$ and $s$, as defined in Appendix A, and let $p_j\in N$ be the associated Calderón-Zygmund projections. Finally, we set
\[a:=a_1-a_2+i(a_3-a_4),\ \ \ \ \ b:=b_1-b_2+i(b_3-b_4),\ \ \ p:=p_1\wedge p_2\wedge p_3\wedge p_4\]
The the decomposition $f=a+b$ clearly satisfies the required conditions with the projection $p\in N$.
\end{proof}

As an immediate consequence, we obtain the following corollary.

\begin{mdframed}[skipabove=10pt]
\begin{coro}\label{singular_kernel_operators_theorem1}
Let $T$ be a bounded operator on $L_2(N)$ given by a singular kernel on $\X$. Then $T$ is of Calderón-Zygmund type with a constant independent of $M$.
\end{coro}
\end{mdframed}
\vspace{5pt}

As the Calderón-Zygmund decomposition does not depend on a choice of a singular kernel, we can extend the conclusion of the above theorem to a larger class of operators.

\begin{mdframed}[skipabove=10pt]
\begin{coro}\label{singular_kernel_operators_theorem2}
Let $d\pg1$, and for $i,j\in\{1,\ldots,d\}$, let $T_{ij}$ be a bounded operator on $L_2(N)$ given by a singular kernel operators on $\X$. Let $T$ be the bounded operator on $L_2(N^{\bar{\oplus d}})=L_2(N)^{\oplus d}$ given by the matrix $(T_{ij})_{ij}$. Then $T$ is of Calderón-Zygmund type with a constant independent of $M$.
\end{coro}
\end{mdframed}
\begin{proof}
Let $f\in(L_1\cap L_2)(N^{\bar{\oplus}d})$ and $s>0$. Then we can write $f=(f_i)_i$ with $f_i\in(L_1\cap L_2)(N)$ for $i\in\{1,\ldots,d\}$. For $i\in\{1,\ldots,d\}$, let
\[f_i=g_i+b_i\]
denote the decomposition associated with $f_i$ and $s$ by Theorem \ref{singular_kernel_operators_CZ}, and let $p_i\in N$ denote the associated projection. Finally, we set 
\[g:=(g_i)_i\in(L_1\cap L_2)(N^{\bar{\oplus}d}),\ \ \ \ \ \ b:=(b_i)_i\in (L_1\cap L_2)(N^{\bar{\oplus}d}),\ \ \ \ \ \ p:=(\wedge _jp_j)_i\in N^{\bar{\oplus} d}.\]
Then the decomposition 
\[f=g+b\]
satisfies the required conditions with the projection $p\in N^{\bar{\oplus} n}$. Indeed, we have
\begin{itemize}[nosep]
    \item $\|g\|_2^2=\sum_i\|g_i\|_2^2\pp C^2s\sum_{i}\|f_i\|_1=C^2s\|f\|_1$,
   \item $\sigma^{\oplus n}(p^{\bot})=d\sigma(\vee_jp_j^{\bot})\pp d\sum_j\sigma(p_j^{\bot})\pp dC^2s^{-1}\sum_j\|f_j\|_1=dC^2s^{-1}\|f\|_1$,
    \item $\|pT(b)p\|_1=\sum_i\big\|(\wedge _jp_j)\big(\sum_jT_{ij}(b_j)\big)(\wedge _jp_j)\big\|_1$
    
    $\pp\sum_{i,j}\|p_jT_{ij}(b_j)p_j\|_1\pp\sum_{i,j}C_{ij}\|f_j\|_1\pp dC\|f\|_1$
\end{itemize}
where $C$ is some constant independent of $M$.
\end{proof}

\section{Singular kernel projections}

Finally, we obtain the following result, which generalizes Bourgain's method to the semicommutative setting.

\begin{mdframed}[skipabove=10pt]
\begin{theo}\label{singular_kernel_projections_theorem1}
Let $P$ be a projection on $L_2(N)$ which satisfies the technical assumptions and which given by a singular kernel on $\X$. Then $P$ is of Jones-type.
\end{theo}
\end{mdframed}
\begin{proof}
Note that $P$ and $P^{\bot}=I-P$ are both given by a singular kernel on $\X$. Thus the result follows from Theorem \ref{CZ_type_theorem1} together with the Marcinkiewciz interpolation theorem, and Theorem \ref{CZ_type_theorem2}. 
\end{proof}

More generally, the following statement holds.

\begin{mdframed}[skipabove=10pt]
\begin{theo}\label{singular_kernel_projections_theorem2}
Let $d\pg1$, and let $P$ be a projection on $L_2(N^{\bar{\oplus}d})$ which satisfies the technical assumptions. Assume that we can write $P=(T_{ij})_{ij}$ where for $i,j\in\{1,\ldots,d\}$, $T_{ij}$ is a bounded operator on $L_2(N)$ given by a singular kernel on $\X$. Then $P$ is of Jones-type.
\end{theo}
\end{mdframed}

\clearpage

\part{Applications to harmonic analysis on the torus}

\section{Preliminaries}


In this chapter, $d\pg1$ is a fixed positive integer. Let $\T^d$ denote the $d$-torus, identified as a subspace of $\C^d$, and let $p:\R^d\to\T^d$ denote the canonical projection, defined by $p(x)=(e^{ix_1},\ldots,e^{ix_d})$ for $x=(x_1,\ldots,x_d)\in\R^d$. The $d$-torus $\T^d$ is equipped with the translation invariant metric $d$ such that
\[d(p(x),p(y))=\inf_{n\in\Z^d}|x-y+2\pi n|,\ \ \ x,y\in\R^d.\]
Finally, the $d$-torus $\T^d$ is equipped with its Haar measure such that 
\[\int_{\T^d}f(z)dz=\frac{1}{(2\pi)^d}\int_{[0,2\pi)^d}f(p(x))dx,\ \ \ \ f\in C(\T^d).\]
Then, the $d$-torus $\T^d$ is a Ahlfors $d$-regular (locally) compact space. In this preliminary section, we provide criteria that allow us to justify that the examples of bounded Fourier multipliers on $L_2(\T^d)$ considered in the following sections are given by singular kernels on $\T^d$.

\begin{mdframed}[skipabove=10pt]
\begin{prop}\label{singular_kernel_torus}
Let $K\in C^1(\R^d\setminus 2\pi\Z^d)$ be a $2\pi\Z^d$-periodic $C^1$-function on $\R^d\setminus 2\pi\Z^d$ such that 
\[|K(y)|\underset{|y|\to0}{=}O(|y|^{-d}),\ \ \ |\nabla K(y)|\underset{|y|\to0}{=}O(|y|^{-(d+1)}).\]
Let $k$ be the kernel on $\T^d$ such that
\[k(p(x),p(y))=K(x-y),\ \ \ \ \ x,y\in\R^d,\ x-y\notin 2\pi\Z^d.\]
Then $k$ satisfies the size condition and the Lipschitz regularity condition (with $\alpha=1$). As a consequence, $k$ is a singular kernel on $\T^d$.
\end{prop}
\end{mdframed}
\begin{proof}
The fact that $k$ satisfies the size condition is obvious. Let $z,w,w'\in\T^d$ such that $2d(w,w')<d(z,w')$. We choose $x,y,y'\in\R^d$ such that $z=p(x)$, $w=p(y)$, $w'=p(y')$, with $d(w,w')=|y-y'|$. By the mean value theorem, we have
\[|K(x-y)-K(x-y')|\pp\sup_{t\in[0,1]}|\nabla K(\gamma(t))|\cdot|y-y'|\]
where $\gamma$ is the smooth path such that
\[\gamma(t)=(1-t)(x-y)+t(x-y')=(x-y)+t(y-y'),\ \ \ \ \ \ \ t\in[0,1].\]
Note that 
\[\inf_{t\in[0,1]}\inf_{n\in\Z^d}|\gamma(t)-2\pi n|\pg\min\big\{d(z,w),d(z,w')\big\}\pg\frac{1}{2}d(z,w').\]
Thus, we have
\[\sup_{t\in[0,1]}|\nabla K(\gamma(t))|\lesssim d(z,w')^{-(d+1)}\]
and
\[|K(x-y)-K(x-y')|\lesssim d(z,w')^{-d}\frac{|y-y'|}{d(z,w')}.\]
Hence, we get
\[|k(z,w)-k(z,w')|\lesssim d(z,w')^{-d}\frac{d(w,w')}{d(z,w')}.\]
Finally, from the estimate
\[V(z,w)=|B_{d(z,w)}(z)|\lesssim d(z,w)^d\]
we obtain
\[|k(z,w)-k(z,w')|\lesssim V(z,w)^{-1}\frac{d(w,w')}{d(z,w')}\]
as desired.
\end{proof}

\begin{mdframed}[skipabove=10pt]
\begin{theo}\label{fourier_multiplier_torus}
Let $m\in L_\infty(\R^d)$ be a $0$-homogeneous bounded symbol on $\R^d$, continuous on $\R^d\setminus\{0\}$, whose inverse Fourier transform, as a distribution on $\R^d\setminus\{0\}$, is a smooth function $K\in C^\infty(\R^d\setminus\{0\})$ such that
\[\int_{\S_\infty^{d-1}}K(x)dx=0\]
where 
\[\S_\infty^{d-1}:=\big\{x\in\R^d\ :\ |x|_\infty:=\sup_{j\in\{1,\ldots,d\}}|x_i|=1\big\}\]
denotes the $(d-1)$-sphere of $\R^d$ w.r.t the norm $|\cdot|_\infty$. Let $\tilde{m}\in C_b(\Z^d)$ be a bounded symbol on $\Z^d$ such that $\tilde{m}(n)=m(n)$ for $n\in\Z^d\setminus\{0\}$, and let $T$ be the bounded Fourier multiplier on $L_2(\T^d)$ with symbol $\tilde{m}$. Then $T$ is given by a kernel on $\T^d$ that satisfies the size condition and the Lipschitz regularity condition (with $\alpha=1$). In particular $T$ is given by a singular kernel on $\T^d$.
\end{theo}
\end{mdframed}
\begin{proof}
First, as $m$ is $0$-homogeneous, we know that $K$ is $(-d)$-homogeneous on $\R^d\setminus\{0\}$, so that we have the following estimates
\[|K(y)|\lesssim|y|^{-d},\ \ |\nabla K(y)|\lesssim|y|^{-(d+1)},\ \ \text{for}\ y\in\R^d\setminus\{0\}.\]
We assert that the expression
\[\tilde{K}(x):=K(x)+\sum_{R=1}^{\infty}\sum_{\protect\substack{m\in\Z^d\\|m|_\infty=R}}K(x+2\pi m),\ \ \ \ \ \text{for}\ x\in\R^d\setminus 2\pi\Z^d\]
defines a $2\pi\Z^d$-periodic continuous function $\tilde{K}\in C(\R^d\setminus 2\pi\Z^d)$. In order to justify this, we can use the mean value theorem to get the estimate
\begin{align*}
\Big|\sum_{|m|_\infty=R}K(x+2\pi m)\Big|&\pp\sum_{|m|_\infty=R}|K(x+2\pi m)-K(2\pi m)|+\Big|\sum_{|m|_\infty=R}K(2\pi m)\Big|\\
&\lesssim\sum_{|m|_\infty=R}\frac{|x|}{ |m|^{d+1}}+\Big|\sum_{|m|_\infty=R}K(2\pi m)\Big|\\
&\lesssim\frac{|x|}{R^{d+1}}\sum_{|m|_\infty=R}1+\Big|\sum_{|m|_\infty=R}K(2\pi m)\Big|\\
&\lesssim\frac{|x|}{R^{d+1}}R^{d-1}+\Big|\sum_{|m|_\infty=R}K(2\pi m)\Big|.
\end{align*}
where we used the obvious asymptotic
\[\sum_{|m|_\infty=R}1\underset{R\to\infty}{=}O(R^{d-1}).\]
As $K$ is $(-d)$-homogeneous, a standard quadrature formula (see for instance \cite{BerlineEulerMacLaurin}) gives the asymptotic
\[R^{-(d-1)}\sum_{|m|_\infty=R}K\big(\frac{m}{R}\big)\underset{R\to\infty}{=}\int_{\S_\infty^{d-1}}K(y)dy+O(R^{-1})\underset{R\to\infty}{=}O(R^{-1}).\]
(For the example to which we want to apply the result, the above sum is actually zero). Thus, again by $(-d)$-homogeneity, we deduce that
\begin{align*}
    \Big|\sum_{|m|_\infty=R}K(2\pi m)\Big|&\lesssim R^{-d}\Big|\sum_{|m|_\infty=R}K\big(\frac{m}{R}\big)\Big|\\
    &=R^{-1}\Big|R^{-(d-1)}\sum_{|m|_\infty=R}K\big(\frac{m}{R}\big)\Big|\underset{R\to\infty}{=}O(R^{-2}).
\end{align*}
Finally, we obtain the estimate
\[\Big|\sum_{|m|_\infty=R}K(x+2\pi m)\Big|\lesssim\frac{|x|+1}{R^2}.\]
Hence $\tilde{K}$ is a well-defined $2\pi\Z^d$-periodic continuous function on $\R^d\setminus 2\pi\Z^d$, and we have
\[|\tilde{K}(x)|\underset{|x|\to0}=O(|x|^{-d}).\]
In addition, $\tilde{K}$ is clearly $C^1$ on $\R^d\setminus2\pi\Z^d$, with
\[\nabla\tilde{K}(x)=\nabla K(x)+\sum_{n\in\Z^d\setminus\{0\}}\nabla K(x+2\pi n),\ \ \ \ \ \text{for}\ x\in\R^d\setminus 2\pi\Z^d\]
(where the sum is uniformly convergent on compacts subsets of $\R^d\setminus 2\pi\Z^d$), so that we also have
\[|\nabla\tilde{K}(x)|\underset{|x|\to0}=O(|x|^{-(d+1)}).\]
Now, the bounded Fourier multiplier on $L_2(\T^d)$ with symbol $\tilde{m}$ is given by the kernel $k$ on $\T^d$ such that 
\[k(p(x),p(y))=\tilde{K}(x-y),\ \ \ \ \ x,y\in\R^d,\ x-y\notin 2\pi\Z^d.\]
By the previous proposition, the kernel $k$ is singular. The proof is complete.
\end{proof}



\section{Spectral projections}

Let $(\chi_n)_{n\in\Z^d}$ denote the canonical Hilbert basis of $L_2(\T^d)$, defined by
\[\chi_n(z)=z^n:=z_1^{n_1}z_2^{n_2}\ldots z_d^{n_d},\ \ \ \ \text{for}\ n\in\Z^d, \ z\in\T^d.\]
Let $M$ be a semifinite von Neumann algebra with trace denoted $\tau$, and let $N:=L_\infty(\T^d)\bar{\otimes}M$ denote the tensor product von Neumann algebra equipped with the tensor product trace denoted $\sigma$. The \textit{Fourier transform} of any $f\in(L_1+L_\infty)(N)$ is the function $\widehat{f}$ on $\Z^d$ which takes values in $(L_1+L_\infty)(M)$ defined by the equations
\[\sigma(f(x\otimes\chi_n))=\tau(\widehat{f}(n)x),\ \ \ \ \text{for}\ n\in\Z^d, \ x\in(L_1\cap L_\infty)(M).\]
Let $I$ be a fixed subset of $\Z^d$, and let $P$ be the \textit{spectral projection} associated with the data $(\T^d,M,I)$, i.e. the projection on $L_2(N)$ such that, if $f\in L_2(N)$ then
\[P(f)=\sum_{n\in I}\chi_n\otimes\widehat{f}(n),\ \ \ \text{in}\ L_2(N).\]
In other words, $P$ is the bounded Fourier multiplier on $L_2(N)$ whose symbol is the characteristic function of $I$. Note that the complement projection $P^{\bot}$ coincides with the spectral projection associated with the data $(\T^d,N,I^{\bot})$ where $I^{\bot}$ is the complement subset of $I$ in $\Z^d$. In the sequel, we denote
\[H(P):=\big\{f\in(L_1+L_\infty)(N)\ :\ \forall n\notin I,\ \widehat{f}(n)=0\big\}\]
so that we clearly have
\[P(L_2(N))=L_2(N)\cap H(P).\]

\begin{lemm}
Let $(K_\lambda)_{\lambda>0}$ denote the Féjer kernel on $N$, i.e. the family of compatible bounded operators $(L_1(N),L_\infty(N))\to(L_1(N),L_\infty(N))$ such that 
\[K_\lambda(f)=\sum_{n\in\Z^d,\ \|n\|_1\pp\lambda}\Big(1-\frac{|n_1|}{\lambda}\Big)\ldots \Big(1-\frac{|n_d|}{\lambda}\Big)\chi_n\otimes\widehat{f}(n),\ \ \ \ \text{for}\ \lambda>0, \ f\in(L_1+L_\infty)(N).\]    
Let $E(N)$ be an exact interpolation space for $(L_1(N),L_\infty(N))$ with order\-/continuous norm. then for every $f\in E(N)$ (resp. $f\in E^{\times}(N)$), the net $(K_\lambda(f))_{\lambda>0}$ converges (in the limit $\lambda\to\infty$) to $f$ in $E(N)$ for the norm topology (resp. in $E^{\times}(N)$ for the \text{\upshape w*}-topology), and thus, for every $f\in E(N)$ and $g\in E^{\times}(N)$, we have
\begin{equation}\label{Féjer_eq}
    \sigma(fg)=\lim_{\lambda\to\infty}\sum_{n\in\Z^d,\ \|n\|_1\pp\lambda}\Big(1-\frac{|n_1|}{\lambda}\Big)^2\ldots \Big(1-\frac{|n_d|}{\lambda}\Big)^2\tau(\widehat{f}(n)\widehat{g}(n)).
\end{equation}
\end{lemm}
\begin{proof}
It is a well-known fact (see \cite{GrafakosClassical}[Remark 3.19]) that $\sup_{\lambda>0}\|K_{\lambda}\|_{L_1(N)\to L_1(N)}\pp 1$. By duality, we deduce that $\sup_{\lambda>0}\|K_{\lambda}\|_{L_\infty(N)\to L_\infty(N)}\pp 1$. By interpolation, we deduce that 
\[\sup_{\lambda>0}\|K_{\lambda}\|_{E(N)\to E(N)}\pp 1,\ \ \ \ \ \ \sup_{\lambda>0}\|K_{\lambda}\|_{E^{\times}(N)\to E^{\times}(N)}\pp 1.\]
As the linear span $\{\chi_n\otimes a\ :\ n\in\Z^d,\ a\in(L_1\cap L_\infty)(M)\}$ is dense in $E(N)$ for the norm topology and dense in $E^{\times}(N)$ for the w*-topology, the desired conclusion easily follows.
\end{proof}

\begin{mdframed}[skipabove=10pt]
\begin{theo}\label{spectral_projection_torus_theorem1}
Let $E(N)$ be an exact interpolation space for $(L_1(N),L_\infty(N))$ with order\-/continuous norm. Then
\[H_E(P)=E(N)\cap H(P),\]
\[H_{E^{\times}}(P)=E^{\times}(N)\cap H(P).\]
Moreover, the subspace
\[(L_1\cap L_\infty)(N)\cap H(P)\]
is norm-dense in $H_E(P)$ and \text{\upshape w*}-dense in $H_{E^{\times}}(P)$.
\end{theo}
\end{mdframed}
\begin{proof}
It is clear that
\[E(N)\cap H(P)=\big\{f\in E(N)\ :\ \forall n\notin I,\ \widehat{f}(n)=0\big\}\] 
is a norm-closed subspace of $E(N)$ that contains $E(N)\cap P(L_2(N))$, and moreover $E(N)\cap P(L_2(N))$ clearly contains $(L_1\cap L_\infty)(N)\cap H(P)$. Thus, it suffices to show that $(L_1\cap L_\infty)(N)\cap H(P)$ is norm-dense in $E(N)\cap H(P)$. Let $f\in E(N)\cap H(P)$. As the net $(K_\lambda(f))_{\lambda>0}$ converges to $f$ in $E(N)$ for the norm topology and also belongs to $E(N)\cap H(P)$, we can assume that there is $g\in E(N)\cap H(P)$ and $\lambda>0$ such that $f=K_\lambda(g)$. Then, we have
\[f=\sum_{n\in I,\ \|n\|_1\pp\lambda}\Big(1-\frac{|n_1|}{\lambda}\Big)\ldots \Big(1-\frac{|n_d|}{\lambda}\Big)\chi_n\otimes\widehat{g}(n).\]
As $(L_1\cap L_\infty)(N)$ is norm-dense in $E(N)$, there is a net $(g_\alpha)_\alpha$ of $(L_1\cap L_\infty)(N)$ that converges to $g$ in $E(N)$ for the norm topology. We set
\[f_\alpha:=\sum_{n\in I,\ \|n\|_1\pp\lambda}\Big(1-\frac{|n_1|}{\lambda}\Big)\ldots \Big(1-\frac{|n_d|}{\lambda}\Big)\chi_n\otimes\widehat{g_\alpha}(n).\]
Then $f_\alpha\in(L_1\cap L_\infty)(N)\cap H(P)$. As the net $(f_\alpha)_\alpha$ clearly converges to $f$ in $E(N)$ for the norm topology, the proof is complete. The dual statement is proved similarly.
\end{proof}

\begin{mdframed}[skipabove=10pt]
\begin{coro}\label{spectral_projection_torus_theorem2}
The spectral projection $P$ satisfies the technical assumptions.
\end{coro}
\end{mdframed}
\begin{proof}
Let $E(N)$ be an exact interpolation space for $(L_1(N),L_\infty(N))$ with order\-/continuous norm. It suffices to show that the polar of $H_E(P)$, $H_{E^{\times}}(P)$ w.r.t. trace duality coincides respectively with $H_{E^{\times}}(P^{\bot})$, $H_E(P^{\bot})$. By Theorem \ref{spectral_projection_torus_theorem1}, it suffices to check that the subspaces $\big\{f\in E(N)\ :\ \forall n\notin I,\ \widehat{f}(n)=0\big\}$ and $\big\{f\in E^{\times}(N)\ :\ \forall n\in I,\ \widehat{f}(n)=0\big\}$ are polar to each other w.r.t. trace duality. This is an easy consequence of \eqref{Féjer_eq}. 
\end{proof}

In the sequel, we assume $d=1$ and $I=\Z_+$. Then $P$ is the \textit{Riesz projection} associated with the data $(\T,M)$. By a classical computation (see \cite{GrafakosClassical}[Chapter 4]), the Riesz projection $P$ is given by the kernel $k$ on $\T$ such that
\[k(z,w)=\frac{w}{w-z},\ \ \ \ z,w\in\T,\ z\neq w.\]
It is obvious that the kernel $k$ satisfies the assumptions of Lemma \ref{singular_kernel_torus}. Indeed, we can write
\[k(p(x),p(y))=\frac{1}{2}+\frac{i}{2}\text{cotan}((x-y)/2),\ \ \ \ \ x,y\in\R,\ x-y\notin 2\pi\Z.\]
Thus $k$ is a singular kernel on $\T$. By Proposition \ref{singular_kernel_projections_theorem1} we obtain the following result.


\begin{mdframed}[skipabove=10pt]
\begin{theo}\label{spectral_projection_torus_theorem4}
The Riesz projection $P$ is of Jones-type.
\end{theo}
\end{mdframed}
\vspace{5pt}

As a result, we recover the well-known interpolation result of Pisier and Xu concerning operator-valued Hardy spaces on the torus. The following lemma is obvious.

\begin{lemm}
Let $O$ be a semifinite von Neumann algebra. Then the amplified projection $Q:=P\otimes I$ on $L_2(N\bar{\otimes}O)=L_2(N)\otimes L_2(O)$ is the Riesz projection associated with the data $(\T,M\bar{\otimes}O)$.
\end{lemm}

\begin{mdframed}[skipabove=10pt]
\begin{theo}\label{spectral_projection_torus_theorem5}
The Riesz projection $P$ satisfies Pisier's method assumptions.
\end{theo}
\end{mdframed}
\begin{proof}
Let $O$ be a semifinite von Neumann algebra. Let $Q:=P\otimes I$ denote the amplified projection on $L_2(N\bar{\otimes}O)=L_2(N)\otimes L_2(O)$. From the above lemma, we already know that $Q$ is of Jones-type. Thus, it suffices to show that $H_\infty(P^{\bot})\odot L_\infty(O)$ is a subspace of $H_\infty(Q^{\bot})$. This is a direct consequence of Theorem \ref{spectral_projection_torus_theorem1}.
\end{proof}

\section{The Leray projection}

Let $C^\infty(\T^d)$ denote the space of smooth functions on $\T^d$ equipped with its canonical locally convex topology. Let $M$ be a semifinite von Neumann algebra with trace denoted $\tau$, and let $N:=L_\infty(\T^d)\bar{\otimes}M$ denote the tensor product von Neumann algebra equipped with the tensor product trace denoted $\sigma$. Let
\[D(N):=C^\infty(\T^d)\odot(L_1\cap L_\infty)(M)\]
denote the algebraic tensor product, identified as a subspace of $(L_1\cap L_\infty)(N)$, and equipped with the projective topology (in the sense of tensor products of locally convex spaces), when $(L_1\cap L_\infty)(M)$ is equipped with its weak topology coming from trace duality with $(L_1+L_\infty)(M)$. By the definition of the projective tensor product, the inclusion $D(N)\to(L_1\cap L_\infty)(N)$ is continuous, where $(L_1\cap L_\infty)(N)$ is equipped with its weak topology coming from trace duality with $(L_1+L_\infty)(N)$. As a consequence, every $f\in(L_1+L_\infty)(N)$ can be considered as an element of the dual space $D'(N)$ by the equation
\[\bra f,g\otimes a\ket=\sigma(f(g\otimes a)),\ \ \ \ g\in C^\infty(\T^d),\ a\in(L_1\cap L_\infty)(M).\]
and the associated injective operator $(L_1+L_\infty)(N)\to D'(N)$ is continuous, when $(L_1+L_\infty)(N)$ is of course equipped with its weak topology coming from trace duality with $(L_1\cap L_\infty)(N)$. For $j\in\{1,\ldots,d\}$, let $\partial_j$ denote the usual partial derivative operator on $L_2(N)$ with respect to the $j$-th coordinate on $\T^d$ (it is the Fourier multiplier on $L_2(N)$ with symbol $\eta_j\in C(\Z^d)$ such that $\eta_j(n)=in_j$ for $n\in\Z^d$). Then $\partial_j$ clearly defines a continuous operator on $D(N)$, and thus by duality we can also consider $\partial_j$ as an operator on the dual space $D'(N)$. 

For $i,j\in\{1,\ldots,d\}$, we consider the bounded function $\rho_{ij}\in C_b(\Z^d)$ such that
\[\rho_{ij}(0)=\delta_{ij},\ \ \ \rho_{ij}(n)=\delta_{ij}-\frac{n_in_j}{|n|^2},\ \ \ \text{for}\ n\in\Z^d\setminus\{0\},\]
and let $R_{ij}$ denote the associated bounded Fourier multiplier on $L_2(N)$. From the obvious relations
\[\overline{\rho_{ij}}=\rho_{ji},\ \ \ \ \ \rho_{ij}=\sum_{k=1}^{d}\rho_{ik}\rho_{kj},\ \ \text{for}\ i,j\in\{1,\ldots,d\}\]
we deduce that the bounded operator $P$ on $L_2(N^{\bar{\oplus}d})=L_2(N)^{\oplus d}$ given by the matrix $(R_{ij})_{i,j}$ is a projection. It is the \textit{Leray projection} associated with the data $(\T^d,M)$. 
In the sequel, we denote
\[H(P):=\Big\{(f_j)_{j}\in(L_1+L_\infty)(N^{\bar{\oplus} d})\ :\ \sum_j\partial_jf_j=0\Big\}\]
and
\[H(P^{\bot}):=\Big\{(f_j)_{j}\in(L_1+L_\infty)(N^{\bar{\oplus}d})\ :\ \forall i,j\in\{1,\ldots,n\},\ \partial_if_j=\partial_jf_i,\ \widehat{f}_j(0)=0\Big\}\]
Then, an easy computation shows that
\[P(L_2(N^{\bar{\oplus}d}))=L_2(N^{\bar{\oplus}d})\cap H(P)\]
and
\[P^{\bot}(L_2(N^{\bar{\oplus}d}))=L_2(N^{\bar{\oplus}d})\cap H(P^{\bot}).\]
If $\lambda>0$ and $f\in(L_1+L_\infty)(N^{\bar{\oplus}d})$, the element $g\in(L_1+L_\infty)(N^{\bar{\oplus}d})$ such that $g_j=K_\lambda(f_j)$ for $j\in\{1,\ldots,n\}$, will be denoted $K_\lambda(f)$. As Fourier multipliers commute with each other, if $f\in H(P)$, then clearly $K_\lambda(f)\in H(P)$ for every $\lambda>0$.

\begin{mdframed}[skipabove=10pt]
\begin{theo}\label{Leray_projection_theorem1}
Let $E$ be an exact interpolation space for $(L_1,L_\infty)$ with order\-/continuous norm. Then
\[H_E(P)=E(N^{\bar{\oplus}d})\cap H(P),\]
\[H_{E^{\times}}(P)=E^{\times}(N^{\bar{\oplus}d})\cap H(P).\]
Moreover, the subspace
\[(L_1\cap L_\infty)(N^{\bar{\oplus}d})\cap H(P)\]
is norm-dense in $H_E(P)$ and \text{\upshape w*}-dense in $H_{E^{\times}}(P)$.
\end{theo}
\end{mdframed}
\begin{proof}
As the identity operator $E(N^{\bar{\oplus}d})\to E(N)^{\oplus d}$ is an isomorphism of normed spaces, it is clear that
\[E(N^{\bar{\oplus}d})\cap H(P)=\Big\{(f_j)_{j}\in E(N^{\bar{\oplus}d})\ :\ \sum_j\partial_jf_j=0\Big\}\] 
is a weakly-closed and thus norm-closed subspace of $E(N^{\bar{\oplus}d})$ that contains $E(N^{\bar{\oplus}d})\cap P(L_2(N^{\bar{\oplus}d}))$, and it is clear that $E(N^{\bar{\oplus}d})\cap P(L_2(N^{\bar{\oplus}d}))$ contains $(L_1\cap L_\infty)(N^{\bar{\oplus}d})\cap H(P)$. As a result, it suffices to show that $(L_1\cap L_\infty)(N^{\bar{\oplus}d})\cap H(P)$ is norm-dense in $E(N^{\bar{\oplus}d})\cap H(P)$. Let $f\in E(N^{\bar{\oplus}d})\cap H(P)$. As the net $(K_\lambda(f))_{\lambda>0}$ converges to $f$ in $E(N^{\bar{\oplus}d})$ for the norm topology and also belongs to $E(N^{\bar{\oplus}d})\cap H(P)$, we can assume that there is $g\in E(N^{\bar{\oplus}d})\cap H(P)$ and $\lambda>0$ such that $f=K_\lambda(g)$. As $(L_1\cap L_\infty)(N^{\bar{\oplus}d})$ is norm-dense in $E(N^{\bar{\oplus}d})$, there is a net $(g_\alpha)_\alpha$ of $(L_1\cap L_\infty)(N^{\bar{\oplus}d})$ that converges to $g$ in $E(N^{\bar{\oplus}d})$ for the norm topology. We set
\[f_\alpha:=K_\lambda(P(g_\alpha))=P(K_\lambda(g_\alpha))\]
Then clearly $f_\alpha\in(L_1\cap L_\infty)(N^{\bar{\oplus}d})\cap H(P)$. Moreover, as the Fourier transform of each coordinate of $K_\lambda(g_\alpha)$ is supported on $\{n\in\Z^d\ :\ |n|_1\pp\lambda\}$, since $P$ is a Fourier multiplier it is clear that the net $(f_\alpha)_\alpha$ converges to $f$ in $E(N)$ for the norm topology, which concludes the proof. The dual statement is proved similarly.
\end{proof}

The following result is proved in a similar way.

\begin{mdframed}[skipabove=10pt]
\begin{theo}\label{Leray_projection_theorem2}
Let $E$ be an exact interpolation space for $(L_1,L_\infty)$ with order\-/continuous norm. Then
\[H_E(P^{\bot})=E(N^{\bar{\oplus}d})\cap H(P^{\bot}),\]
\[H_{E^{\times}}(P^{\bot})=E^{\times}(N^{\bar{\oplus}d})\cap H(P^{\bot}).\]
Moreover, the subspace
\[(L_1\cap L_\infty)(N^{\bar{\oplus}d})\cap H(P^{\bot})\]
is norm-dense in $H_E(P^{\bot})$ and \text{\upshape w*}-dense in $H_{E^{\times}}(P^{\bot})$.
\end{theo}
\end{mdframed}

\begin{mdframed}[skipabove=10pt]
\begin{coro}\label{Leray_projection_theorem3}
The Leray projection $P$ satisfies the technical assumptions.
\end{coro}
\end{mdframed}
\begin{proof}
It easily follows from the previous theorems. Details are left to the reader.
\end{proof}

\begin{lemm}
If $i,j\in\{1,\ldots,d\}$, the bounded operator $R_{ij}$ is given by a singular kernel on $\T^d$.
\end{lemm}
\begin{proof}
Let $m_{ij}\in L_\infty(\R^d)$ be the bounded symbol on $\R^d$ such that 
\[m_{ij}(\xi)=\delta_{ij}-\frac{\xi_i\xi_j}{|\xi|^2},\ \ \ \text{for}\ \xi\in\R^d\setminus\{0\}.\]
By a classical computation (see \cite{Mitrea}[Proposition 4.70]), the inverse Fourier transform $K_{ij}$ of $m_{ij}$, as a distribution on $\R^d\setminus\{0\}$, is a smooth function $K_{ij}\in C^\infty(\R^d\setminus\{0\})$ such that
\[K_{ij}(x)=-\frac{1}{\omega_{d-1}}\frac{\partial}{\partial x_i}\Big[\frac{x_j}{|x|^{d}}\Big]=\frac{\delta_{ij}}{\omega_{d-1}}\frac{1}{|x|^d}-\frac{d}{\omega_{d-1}}\frac{x_ix_j}{|x|^{d+1}},\ \ \ \ x\in\R^d\setminus\{0\},\]
where $\omega_{d-1}$ denotes the surface measure of the $(d-1)$-sphere. We have
\[\int_{\S^{d-1}_\infty}K_{ij}(y)dy=0.\]
Indeed, if $i\neq j$ then $K_{ij}((y_1,\ldots,-y_i,\ldots,y_d))=-K_{ij}(y)$ for $y\in\R^d\setminus\{0\}$, so that the above integral clearly vanishes. In the case $i=j$, as the sphere $\S_\infty^{d-1}$ is invariant under any permutations of the coordinates, we have
\begin{align*}
 \int_{\S_\infty^{d-1}}K_{ii}(y)dy&=\frac{1}{\omega_{d-1}}\Big[\int_{\S_\infty^{d-1}}|y|^{-d}dy-d\int_{\S_\infty^{d-1}}y_i^2|y|^{-(d+2)}dy\Big]\\
 &=\frac{1}{\omega_{d-1}}\Big[\int_{\S_\infty^{d-1}}|y|^{-d}dy-\sum_{k=1}^{d}\int_{\S_\infty^{d-1}}y_k^2|y|^{-(d+2)}dy\Big]\\
  &=\frac{1}{\omega_{d-1}}\Big[\int_{\S_\infty^{d-1}}|y|^{-d}dy-\int_{\S_\infty^{d-1}}|y|^2|y|^{-(d+2)}dy\Big]=0.
\end{align*}
Hence, by a direct application of Theorem \ref{fourier_multiplier_torus} we get the desired conclusion.
\end{proof}

From Theorem \ref{singular_kernel_projections_theorem2}, we obtain the following result.

\begin{mdframed}[skipabove=10pt]
\begin{theo}\label{Leray_projection_theorem4}
The Leray projection $P$ and its complement $P^{\bot}$ are of Jones-type.
\end{theo}
\end{mdframed}
\vspace{5pt}

The following lemma is obvious.

\begin{lemm}
Let $O$ be a semifinite von Neumann algebra. Then the amplified projection $Q:=P\otimes I$ on $L_2(N\bar{\otimes}O)=L_2(N)\otimes L_2(O)$ is the Leray projection associated with the data $(\T^n,M\bar{\otimes}O)$.
\end{lemm}

\begin{mdframed}[skipabove=10pt]
\begin{theo}
The Leray-projection $P$ and its complement $P^{\bot}$ satisfy Pisier's method assumptions.
\end{theo}
\end{mdframed}
\begin{proof}
Let $O$ be a semifinite von Neumann algebra. Let $Q:=P\otimes I$ denote the amplified projection on $L_2(N\bar{\otimes}O)=L_2(N)\otimes L_2(O)$. From the above lemma, we already know that $Q$ and $Q^{\bot}$ are of Jones-type. Thus, it suffices to show that $H_\infty(P)\odot L_\infty(O)$ and $H_\infty(P^{\bot})\odot L_\infty(O)$ are subspaces of $H_\infty(Q)$ and $H_\infty(Q^{\bot})$ respectively. This is a direct consequence of Theorem \ref{Leray_projection_theorem1} and Theorem \ref{Leray_projection_theorem2}.
\end{proof}

If $E$ is an exact interpolation space for $(L_1,L_\infty)$, we define the \textit{Sobolev space} $W_{1,E}(N)$ as the space of $g\in D'(N)$ such that $\partial_j g\in E(N)$ for every $j\in\{1,\ldots,d\}$, and equipped with the norm $\|\cdot\|_{W_{1,E}(N)}$ such that
\[\|g\|_{W_{1,E}(N)}=\|(\partial_jg)_j\|_{E(N^{\bar{\oplus}d})},\ \ \ \ \ \ \ \ \ \text{for}\ g\in W_{1,E}(N).\]
\[\|g\|_{W_{1,E}(N)}=\left\{\begin{array}{cl}
     \Big[\sum_{j=1}^{d}\|\partial_jg\|_{L_p(N)}^p\Big]^{1/p} & \text{if}\ p<\infty \\
     \sup_{j\in\{1,\ldots,d\}}\|\partial_jg\|_{L_\infty(N)} & \text{if}\ p=\infty
\end{array}\right.,\ \ \ \ \ \ \ \ \ \text{for}\ g\in W_{1,p}(N).\]
We also denote $W_{1,p}(N):=W_{1,L_p}(N)$ for $1\pp p\pp\infty$. By using the Féjer kernel, we easily get the following proposition.

\begin{prop}
Let $E$ be an exact interpolation space for $(L_1,L_\infty)$ with order\-/continuous norm. Then the elements of $D'(N)$ of the form $\chi_n\otimes a$ with $n\in\Z^d$ and $a\in(L_1\cap L_\infty)(M)$ span a norm-dense subspace of $W_{1,E}(N)$ and a \text{\upshape w*}-dense subspace of $W_{1,E^{\times}}(N)$.
\end{prop}

An easy computation yields the following lemma.

\begin{lemm}
Let $f\in(L_1+L_\infty)(N^{\bar{\oplus}d})$. Then $f\in H(P^{\bot})$ if and only if there is $g\in D'(N)$ such that $f_j=\partial_jg$ for every $j\in\{1,\ldots,d\}$.
\end{lemm}

As a consequence of the previous lemma, if $E$ is an exact interpolation space for $(L_1,L_\infty)$ with order\-/continuous norm, the operators $W_{1,E}(N)\to E(N^{\bar{\oplus} d})$, $g\mapsto(\partial_jg)_j$ and $W_{1,E^{\times}}(N)\to E^{\times}(N^{\bar{\oplus} d})$, $g\mapsto(\partial_jg)_j$ are isometric embeddings of normed spaces, with range $H_E(P^{\bot})$ and $H_{E^{\times}}(P^{\bot})$ respectively. As a consequence of the previous theorems, we directly deduce the following results.

\begin{mdframed}[skipabove=10pt]
\begin{theo}
Let $1\pp p_0,p_1\pp\infty$ and $f\in W_{1,p_0}(N)+W_{1,p_0}(N)$. Then, for every $t>0$, up to a constant depending only on $p_0,p_1,d$, we have
\[K_t(f,W_{1,p_0}(N),W_{1,p_1}(N))\ \simeq\ \sum_{j=1}^{d}K_t(\partial_jf,L_{p_0}(N),L_{p_1}(N)).\]
\end{theo}
\end{mdframed}
\vspace{5pt}

\begin{mdframed}[skipabove=10pt]
\begin{theo}
Let $1\pp p_0,p_1\pp\infty$, and let $\Phi$ be a $K$-parameter space such that the exact interpolation space
\[E:=K_\Phi(L_{p_0},L_{p_1})\]
has order-continuous norm. Then 
\[W_{1,E}(N)=K_\Phi(W_{1,p_0}(N),W_{1,p_1}(N))\]
with equivalent norms, with constants depending only on $p_0,p_1,d$. 
\end{theo}
\end{mdframed}
\vspace{5pt}

\begin{mdframed}[skipabove=10pt]
\begin{theo}
For every $1\pp p_0,p_1<\infty$ and $0<\theta<1$, we have
\[[W_{1,p_0}(N),W_{1,p_1}(N)]_{\theta}=W_{1,p_\theta}(N)\]
with equivalent norms, with constants depending only on $p_0,p_1,\theta,d$, where 
\[\frac{1}{p_\theta}=\frac{1-\theta}{p_0}+\frac{\theta}{p_1}.\]
\end{theo}
\end{mdframed}
\vspace{5pt}

\clearpage

\part{Appendix : semicommutative Calderón-Zygmund decomposition}

In this paragraph, we define the noncommutative Calderón-Zygmund decomposition as introduced by J. Parcet, J.M. Conde Alonso and L. Cadhilac in \cite{CadilhacCZ1}.

Let $\X$ be a Ahlfors $d$-regular locally compact space of homogeneous type. We recall that this means that $\X$ is a locally compact space equipped with a Radon measure and a metric that induces its topology, for which the open balls of finite radius are relatively compact, and we have the Ahlfors $d$-regularity condition
\[|B_r(x)|\simeq r^d,\ \ \ \ \ \text{for}\ \ \ x\in\X,\ 0<r<\diam(\X).\]
Here $|\cdot|$ refers to the volume w.r.t. the measure and $B$ refers to the open balls w.r.t. the metric, and the notation $\simeq$ means that the estimate holds, up to some constants depending only on $\X$. In particular, the metric of $\X$ satisfies the following generalised doubling condition
\[|B_{\lambda r}(x)|\lesssim\lambda^d|B_r(x)|,\ \ \ \text{for}\ x\in\X,\ r>0,\ \lambda\pg1\]
and the following reverse doubling condition for small balls
\[|B_{\lambda r}(x)|\gtrsim\lambda^{d}|B_r(x)|,\ \ \ \text{for}\ x\in\X,\ r>0,\ \lambda\pp1.\]
In particular $\X$ is a space of homogeneous type in the usual sense of Coifman and Weiss \cite{CoifmanWeiss}. For $x,y\in\X$, we set
\[V(x,y):=|B_{d(x,y)}(x)|.\]
The following estimates are classical.

\begin{lemm}
If $x,y\in\X$, then
\[V(x,y)\lesssim V(y,x).\]
\end{lemm}
\begin{proof}
If $r\pg d(x,y)$ then $B_r(x)$ is clearly included in $B_{2r}(y)$, thus by the doubling condition we get $|B_r(x)|\lesssim|B_\delta(y)|$. The choice $r=d(x,y)$ gives the desired result.
\end{proof}

\begin{lemm}
If $y\in\X$, and $\delta,\alpha>0$ then
\begin{equation}\label{lemma_doubling_estimate_0}
\int_{\X\setminus B_r(y)}\frac{d(x,y)^{-\alpha}}{V(x,y)}dx\lesssim\frac{r^{-\alpha}}{2^\alpha-1}.
\end{equation}
\end{lemm}
\begin{proof}
\begin{align*}
    \int_{\X\setminus B_r(y)}\frac{d(x,y)^{-\alpha}}{V(x,y)}dx&=\sum_{m=0}^{\infty}\int_{B_{2^{m+1}r}(y)\setminus B_{2^mr}(y)}\frac{d(x,y)^{-\alpha}}{V(x,y)}dx\\
    &\lesssim\sum_{m=0}^{\infty}[2^mr]^{-\alpha}\int_{B_{2^{m+1}r}(y)\setminus B_{2^{m}r}(y)}\frac{dx}{V(y,x)}\\
    &\pp\sum_{m=0}^{\infty}[2^mr]^{-\alpha}\frac{1}{|B_{2^mr}(y)|}\int_{B_{2^{m+1}r}(y)\setminus B_{2^{m}r}(y)}dx\\
    &\pp\sum_{m=0}^{\infty}[2^mr]^{-\alpha}\frac{|B_{2^{m+1}r}(y)|}{|B_{2^mr}(y)|}\\
    &\lesssim\sum_{m=0}^{\infty}[2^mr]^{-\alpha}=\frac{r^{-\alpha}}{1-2^{-\alpha}}.
\end{align*}
\end{proof}

Let $M$ be a semifinite von Neumann algebra with trace denoted $\tau$. Let $N:=L_\infty(\X)\bar{\otimes}M$ denote the tensor product von Neumann algebra equipped with the tensor product trace denoted $\sigma$. Then $L_1(N)=L_1(\X,L_1(M))$ coincide with the Bochner space of $L_1(M)$-valued integrable functions on $\X$. The Bochner integral of $f\in L_1(N)$ will be denoted 
\[\int_\X f(x)dx\in L_1(M).\]
For $n\pg0$, let $\mathcal{Q}_n$ denote the collection of Christ cubes of positive order $n$ as introduced in \cite{Christ} (note that, in the original construction, cubes with negative orders are also considered, but for our purpose we will not have to deal with any of them). For every $n\pg1$, $Q\in\mathcal{Q}_n$, and $f\in L_1(N)$, let denote 
\[f_Q:=\frac{1}{|Q|}\int_Qf(x)dx=\frac{1}{|Q|}\int_{\X}[(1_Q\otimes 1)f(1_Q\otimes 1)](x)dx.\]
The following lemma is taken from \cite{Christ}[Theorem 11]. 

\begin{lemm}\label{lemma_ChristCubes}
The following assertions hold.
\begin{enumerate}[nosep]
    \item If $n\pg0$ then $\mathcal{Q}_n$ is a measurable partition of $\X$ (i.e. the elements of $\mathcal{Q}_n$ are disjoints and $X\setminus\cup_{Q\in\mathcal{Q}_n}Q$ is negligible).
    \item If $n\pg0$ and $Q\in\mathcal{Q}_n$, then $Q$ contains $B_{C_1\delta^n}(y_Q)$, for some $y_Q\in Q$, where $C_1$ is a positive constant.
    \item If $n\pg0$ and $Q\in\mathcal{Q}_n$, then $d_Q\leq C_2\delta^n$ where $d_Q$ denotes the diameter of $Q$, and where $C_2$ is a positive constant.
    \item If $n\pg1$, and $Q\in\mathcal{Q}_n$, then there is a unique $\hat{Q}\in\mathcal{Q}_{n-1}$ such that $Q\subset\hat{Q}$.
\end{enumerate}
\end{lemm}

For every $n\pg1$, let $\X_n$ denote the measure space obtained by equipping $\X$ with its $\sigma$-algebra generated by $\mathcal{Q}_n$, and let denote $N_n:=L_\infty(\X_n)\bar{\otimes}M$. By the properties of Christ cubes, the sequence $(N_n)_{n\pg1}$ is a filtration on $N$. Let $(E_n)_{n\pg1}$ denote the associated conditional expectations. Then, for every $f\in(L_1\cap L_2)(N)$ and $n\pg1$, we have
\[E_n(f)=\sum_{Q\in\mathcal{Q}_n}1_Q\otimes f_Q\ \ \ \ \text{in}\ (L_1\cap L_2)(N).\]
Let $f\in(L_1\cap L_2)(N)_+$ and $s>0$ be fixed. Let $(q_n)_{n\pg1}$ denote the associated sequence of so-called Cuculescu's projections. It is defined by induction as follows,
\[q_n:=1_{[0, s]}(q_{n-1}E_n(f)q_{n-1}),\ \ \ n\pg1,\]
with the convention $q_0:=1$. We set
\[q:=\wedge_{n\pg1}q_n.\]
For $n\pg1$, we consider the projection
\[e_n:=q_{n-1}-q_n\]
with the convention $q_0=1$, and we set 
\[e:=\sum_{n\pg1}e_n=1-q.\]
Finally, we set
\[a:=qfq+\sum_{n\pg1}E_{n-1}(e_nfe_n),\]
\[b_d:=\sum_{n\pg1}\big[e_nfe_n-E_{n-1}(e_nfe_n)\big]\]
and 
\[b_o:=f-(a+b_d)=f-qfq-\sum_{n\pg1}e_nfe_n.\]
As $f\in L_1(N)_+$, it is clear that the series above are absolutely convergent in $L_1(N)$, so that $a,b_o,b_d$ are well-defined elements of $L_1(N)$.

\begin{mdframed}[backgroundcolor=black!10,rightline=false,leftline=false,topline=false,bottomline=false,skipabove=10pt]
The decomposition
\[f=a+b\]
where
\[b:=b_d+b_o\]
is the \textit{Calderón-Zygmund decomposition} associated with $f$ and $s$ as introduced in \cite{CadilhacCZ2}. 
\end{mdframed}

\begin{rem}
In \cite{CadilhacCZ2} two different Calderón-Zygmund decompositions. The one we use corresponds to the second one in the cited article, even if it is a bit more complicated than the first one. The reason is, for the first decomposition, the proof of one crucial estimate requires the assumption $\|E_n(f)\|_\infty\underset{n\to\infty}{\to}0$, which is not valid anymore in full generality when the measure does not satisfy the reverse doubling condition for balls with big radius. 
\end{rem}

The proof of the estimates that follow are essentially already contained in \cite{CadilhacCZ2}. For conveniance we give the details.

\begin{mdframed}[skipabove=10pt]
\begin{theo}
We have $a,b_d,b_o\in(L_1\cap L_2)(N)$ with the estimate
\[\|a\|_2^2\lesssim s\|f\|_1.\]
\end{theo}
\end{mdframed}
\begin{proof}
It is essentially proved in \cite{CadilhacCZ2}[Lemma 1.2].
\end{proof}

For every $n\pg1$ and $Q\in\mathcal{Q}_n$, we set
\[e_Q:=\frac{1}{|Q|}\int_Qe_n(x)dx,\ \ \ \ \ \ \ \ \ \ \ \ q_Q:=\frac{1}{|Q|}\int_Qq_n(x)dx.\]
Then, if $n\pg1$ we have
\[e_n=\sum_{Q\in\mathcal{Q}_n}1_Q\otimes e_Q,\ \ \ \ \ \ \ \ \ \ q_n=\sum_{Q\in\mathcal{Q}_n}1_Q\otimes q_Q.\]

\begin{mdframed}[backgroundcolor=black!10,rightline=false,leftline=false,topline=false,bottomline=false,skipabove=10pt]
The projection
\[p:=\Big(\bigvee_{n\pg1}\bigvee_{Q\in\mathcal{Q}_n}1_{B_{3d_Q}(y_Q)}\otimes e_Q\Big)^{\bot}\]
is the \textit{Calderón-Zygmund projection} associated with $f$ and $s$. (we recall that $d_Q$ are $y_Q$ are defined in Lemma \ref{lemma_ChristCubes}).
\end{mdframed}

\begin{mdframed}[skipabove=10pt]
\begin{theo}\label{Cadhilac_estimate_p}
We have the estimate $\sigma(p^{\bot})\lesssim s^{-1}\|y\|_1$.
\end{theo}
\end{mdframed}

\vspace{5pt}

\begin{lemm}
If $n\pg1$ and $Q\in\mathcal{Q}_n$, then we have $|B_{3d_Q}(y_Q)|\lesssim|Q|$.
\end{lemm}
\begin{proof}
This is a direct consequence of Lemma \ref{lemma_ChristCubes} together with the doubling condition.
\end{proof}

\begin{proof}[Proof of Theorem \ref{Cadhilac_estimate_p}]
By using the previous lemma, we have
\begin{align*}
\sigma(p^{\bot})&=\sigma\Big(\bigvee_{n\pg1}\bigvee_{Q\in\mathcal{Q}_n}1_{B_{3d_Q}(y_Q)}\otimes e_Q\Big)\pp\sum_{n\pg1}\sum_{Q\in\mathcal{Q}_n}\sigma(1_{B_{3d_Q}(y_Q)}\otimes e_Q)\\
&=\sum_{n\pg1}\sum_{Q\in\mathcal{Q}_n}|B_{3d_Q}(y_Q)|\tau(e_Q)\lesssim\sum_{n\pg1}\sum_{Q\in\mathcal{Q}_n}|Q|\tau(e_Q)\\
&=\sum_{n\pg1}\sum_{Q\in\mathcal{Q}}\sigma(1_Q\otimes e_Q)=\sum_{n\pg1}\sigma(e_n)=\sigma(e)=\sigma(1-q)\pp s^{-1}\|y\|_1
\end{align*}
where the last estimate comes from the well-known properties of Cuculescu's projections.
\end{proof}

Let $k$ be a \textit{singular kernel} on $\X$, that is a scalar-valued continuous function 
\[\fonctbis{\{(x,y)\in \X\times\X\ :\ x\neq y\}}{\C}{(x,y)}{k(x,y)}\]
that satisfies the \textit{size condition}
\[|k(x,y)|\lesssim\frac{1}{V(x,y)}\]
for every $x,y\in\X$, $x\neq y$, and the \textit{$L_2$-H\"ormander regularity condition}
\[\sum_{m=0}^{\infty}|B_{2^{m+1}r}(y')\setminus B_{2^mr}(y')|^{1/2}\sup_{y\in B_{r/3}(y')}\Big[\int_{B_{2^{m+1}r}(y')\setminus B_{2^mr}(y')}\big|k(x,y)-k(x,y')|^2dx\Big]^{1/2}\lesssim1\]
for every $y'\in\X$ and $r>0$. This condition is taken from \cite{CadilhacCZ2}. As the following proposition shows, it is stronger than the usual H\"ormander regularity condition.

\begin{prop}\label{annexe_hormander_estimate}
For every $y\in\X$ and $r>0$, we have
\begin{equation*}\label{hormander_estimate}
    \sup_{y\in B_{r/3}(y')}\int_{\X\setminus B_{r}(y)}|k(x,y)-k(x,y')|dx\lesssim1\ \ \ \ \ \ \ \ \ \textit{(H\"ormander's regularity condition)}
\end{equation*}
\end{prop}
\begin{proof}
By H\"older's inequality, we have
\begin{align*}
&\sup_{y\in B_{r/3}(y')}\int_{\X\setminus B_{r}(y)}|k(x,y)-k(x,y')|dx\\
&\pp\sup_{y\in B_{r/3}(y')}\sum_{m=0}^{\infty}\int_{B_{2^{m+1}r}(y')\setminus B_{2^mr}(y')}|k(x,y)-k(x,y')|dx\\
&\pp\sup_{y\in B_{r/3}(y')}\sum_{m=0}^{\infty}|B_{2^{m+1}r}(y')\setminus B_{2^mr}(y')|^{1/2}\Big[\int_{B_{2^{m+1}r}(y')\setminus B_{2^mr}(y')}|k(x,y)-k(x,y')|^2dx\Big]^{1/2}\\
&\pp\sum_{m=0}^{\infty}|B_{2^{m+1}r}(y')\setminus B_{2^mr}(y')|^{1/2}\sup_{y\in B_{r/3}(y')}\Big[\int_{B_{2^{m+1}r}(y')\setminus B_{2^mr}(y')}|k(x,y)-k(x,y')|^2dx\Big]^{1/2}\\
&\lesssim1.
\end{align*}
\end{proof}


Let $T$ be a bounded operator on $L_2(\X)$ that admits the following integral represention
\[(Tf)(x)=\int_{\X}k(x,y)f(y)dy\]
for every $f\in L_2(\X)$ with compact support and almost every $x\in\X$ outside the support of $f$. Then, it is easy to see that the amplified bounded operator $T\otimes I$ on $L_2(N)=L_2(\X)\otimes L_2(M)$ admits the following integral represention
\[(T\otimes I)(f)(x)=\int_{\X}k(x,y)f(y)dy\]
for every $f\in L_2(N)$ with compact support and almost every $x\in\X$ outside the support of $f$. 

\begin{mdframed}[skipabove=10pt]
\begin{theo}\label{Cadhilac_estimate_diag}
We have the estimate
\[\|p(T\otimes I)(b_d)p\|_1\lesssim\|f\|_1.\]
\end{theo}
\end{mdframed}

\vspace{5pt}

For the proof of the above result, we need a couple of lemmas. First, for every $n\pg1$ and $Q\in\mathcal{Q}_n$, we set 
\[b_d^Q:=(1_Q\otimes e_Q)f(1_Q\otimes e_Q)-|\hat{Q}|^{-1}1_{\hat{Q}}\otimes |Q|e_Qf_Qe_Q\]
where $\hat{Q}\in\mathcal{Q}_{n-1}$ is defined in Lemma \ref{lemma_ChristCubes}.

\begin{lemm}
We have the estimate
\[b_d=\sum_{n\pg1}\sum_{Q\in\mathcal{Q}_n}b_d^Q\ \ \ \ \ \text{in}\ L_2(N)\]
and
\[\sum_{n\pg1}\sum_{Q\in\mathcal{Q}_n}\|b_d^Q\|_1\pp2\|f\|_1.\]
\end{lemm}
\begin{proof}
We have
\[b_d=\sum_{n\pg1}\big[e_nfe_n-E_{n-1}(e_nfe_n)\big]\ \ \ \ \ \text{in}\ \ L_2(N).\]
An easy computation shows that
\[e_nfe_n-E_{n-1}(e_nfe_n)=\sum_{Q\in\mathcal{Q}}b_d^Q\\ \ \ \ \ \text{in}\ \ (L_1\cap L_2)(N)\]
and
\[\sum_{Q\in\mathcal{Q}_n}\|b_d^Q\|_1\pp 2\|e_nfe_n\|_1.\]
The lemma directly follows.
\end{proof}

\begin{lemm}
Let $n\pg1$ and $Q\in\mathcal{Q}_n$. Then 
\[\int_{\X\setminus B_{3d_{\hat{Q}}}(y_{\hat{Q}})}\|(T\otimes I)(b_d^Q)(x)\|_1dx\lesssim\|b_d^Q\|_1.\]
\end{lemm}
\begin{proof}
As the support of $b_d^Q$ is included in $\hat{Q}$, and since $\int_\X b_d^Q(y)dy=0$, for $x\in\X\setminus B_{3d_{\hat{Q}}}(y_{\hat{Q}})$ we can use the kernel representation to write
\begin{align*}
(T\otimes I)(b_d^Q)(x)&=\int_\X k(x,y)b_d^Q(y)dy\\
&=\int_\X\big[k(x,y)-k(x,y_Q)\big]b_d^Q(y)dy\\
\end{align*}
so that
\[\|(T\otimes I)(b_d^Q)(x)\|_1\pp\int_{B_{d_{\hat{Q}}(y_{\hat{Q}})}}|k(x,y)-k(x,y_Q)|\|b_d^Q(y)\|_1dy.\]
Thus, we obtain
\begin{align*}
\int_{\X\setminus B_{3d_{\hat{Q}}}(y_{\hat{Q}})}\|(T\otimes I)(b_d^Q)(x)\|_1dx&\pp\int_{\X\setminus B_{3d_{\hat{Q}}}(y_{\hat{Q}})}\Big[\int_{B_{d_{\hat{Q}}(y_{\hat{Q}})}}|k(x,y)-k(x,y_Q)|\|b_d^Q(y)\|_1dy\Big]dx\\
&=\int_{B_{d_{\hat{Q}}(y_{\hat{Q}})}}\Big[\int_{\X\setminus B_{3d_{\hat{Q}}}(y_{\hat{Q}})}|k(x,y)-k(x,y_Q)|dx\Big]\|b_d^Q(y)\|_1dy\\
&\lesssim\int_{\hat{Q}}\|b_d^Q(y)\|_1dy=\|b_d^Q\|_1
\end{align*}
where we used H\"ormander's regularity condition \eqref{hormander_estimate} in the last estimate.
\end{proof}

\begin{lemm}
Let $Q\in\mathcal{Q}$. Then 
\[\int_{B_{3d_{\hat{Q}}}(y_{\hat{Q}})\setminus B_{3d_Q}(y_Q)}\|(T\otimes I)(b_d^Q)(x)\|_1dx\lesssim\int_Q\|e_Qf(x)e_Q\|_1dx.\]
\end{lemm}
\begin{proof}
We set 
\[c_d^Q:=|\hat{Q}|^{-1}1_{\hat{Q}}\otimes |Q|e_Qf_Qe_Q=|\hat{Q}|^{-1}1_{\hat{Q}}\otimes\int_Qe_Qf(x)e_Qdx\]
so that 
\[b_d^Q=(1_Q\otimes e_Q)f(1_Q\otimes e_Q)-c_d^Q.\]
As the support of $(1_Q\otimes e_Q)f(1_Q\otimes e_Q)$ is clearly included in $Q$, for $x\in B_{3d_{\hat{Q}}}(y_{\hat{Q}})\setminus B_{3d_Q}(y_Q)$ we can use the integral representation to write
\[(T\otimes I)(b_d^Q)(x)=\int_{Q}k(x,y)e_Qf(y)e_Qdy-(T\otimes I)(c_d^Q)(x)\]
and by the size condition on the kernel, we obtain the estimate
\[\|(T\otimes I)(b_d^Q)(x)\|_1\lesssim\int_{Q}\frac{1}{V(x,y)}\|e_Qf(y)e_Q\|_1dy+\|(T\otimes I)(c_d^Q)(x)\|_1.\]
One the one hand, by \eqref{lemma_doubling_estimate_0} we have
\begin{align*}
&\int_{B_{3d_{\hat{Q}}}(y_{\hat{Q}})\setminus B_{3d_Q}(y_Q)}\Big[\int_{Q}\frac{1}{V(x,y)}\|e_Qf(y)e_Q\|_1dy\Big]dx\\
&=\int_{Q}\Big[\int_{B_{3d_{\hat{Q}}}(y_{\hat{Q}})\setminus B_{3d_Q}(y_Q)}\frac{dx}{V(x,y)}\Big]\|e_Qf(y)e_Q\|_1dy\\
&\pp\int_{Q}\Big[\int_{B_{4d_{\hat{Q}}}(y)\setminus B_{2d_Q}(y)}\frac{dx}{V(x,y)}\Big]\|e_Qf(y)e_Q\|_1dy\\
&\pp 4d_{\hat{Q}}\int_{Q}\Big[\int_{\X\setminus B_{2d_Q}(y)}\frac{d(x,y)^{-1}}{V(x,y)}dx\Big]\|e_Qf(y)e_Q\|_1dy\\
&\lesssim 3d_{\hat{Q}}[2d_{Q}]^{-1}\int_{Q}\|e_Qf(y)e_Q\|_1dy\\
&\lesssim\int_{Q}\|e_Qf(y)e_Q\|_1dy
\end{align*}
where the last estimate comes from Lemma \ref{lemma_ChristCubes} together with Ahlfors $d$-regularity. One the other hand, by the very definition of $c_d^Q$, we have
\[(T\otimes I)(c_d^Q)(x)=T\big[|\hat{Q}|^{-1}1_{\hat{Q}}\big](x)\int_Qe_Qf(x)e_Qdx\]
so that
\begin{align*}
&\int_{B_{3d_{\hat{Q}}}(y_{\hat{Q}})}\|(T\otimes I)(c_d^Q)(x)\|_1dx=\int_{B_{3d_{\hat{Q}}}(y_{\hat{Q}})}\big|T\big[|\hat{Q}|^{-1}1_{\hat{Q}}\big](x)\big|dx\Big\|\int_Qe_Qf(x)e_Qdx\Big\|_1\\
&\pp|B_{3d_{\hat{Q}}}(y_{\hat{Q}})|^{1/2}\Big[\int_{B_{3d_{\hat{Q}}}(y_{\hat{Q}})}\big|T\big[|\hat{Q}|^{-1}1_{\hat{Q}}\big](x)\big|^2dx\Big]^{1/2}\int_{Q}\|e_Qf(x)e_Q\|_1dx\\
&\pp|B_{3d_{\hat{Q}}}(y_{\hat{Q}})|^{1/2}\|T\|_{L_2(\X)\to L_2(\X)}\Big[\int_{B_{3d_{\hat{Q}}}(y_{\hat{Q}})}\big|\big[|\hat{Q}|^{-1}1_{\hat{Q}}\big](x)\big|^2dx\Big]^{1/2}\int_{Q}\|e_Qf(x)e_Q\|_1dx\\
&=|B_{3d_{\hat{Q}}}(y_{\hat{Q}})|^{1/2}\|T\|_{L_2(\X)\to L_2(\X)}|\hat{Q}|^{-1/2}\int_{Q}\|e_Qf(x)e_Q\|_1dx\\
&\lesssim\int_{Q}\|e_Qf(x)e_Q\|_1dx
\end{align*}
where the last estimate is again a consequence of Lemma \ref{lemma_ChristCubes} together with Ahlfors $d$-regularity. By combining the previous estimates, we get the desired conclusion.
\end{proof}

\begin{lemm}\label{Cadhilac_lemma_support}
Let $n\pg1$ and $Q\in\mathcal{Q}_n$. Then the support of $p(T\otimes I)(b_d^Q)p$ is included in $\X\setminus B_{3d_Q}(y_Q)$. 
\end{lemm}
\begin{proof}
By the very definition of $b_d^Q$, we have $b_d^Q=(1\otimes e_Q)b_d^Q(1\otimes e_Q)$, so $(T\otimes I)(b_d^Q)=(1\otimes e_Q)(T\otimes I)(b_d^Q)(1\otimes e_Q)$, i.e. $(T\otimes I)(b_d^Q)(x)=e_Q(T\otimes I)(b_d^Q)(x)e_Q$ for $x\in\X$. Now, if $x\in B_{3d_Q}(y_Q)$, then by definition we know that $p(x)$ and $e_Q$ are orthogonal, and thus we get $p(x)(T\otimes I)(b_d^Q)(x)p(x)=0$ as desired.
\end{proof}

\begin{proof}[Proof of Theorem \ref{Cadhilac_estimate_diag}]
Let $n\pg1$ and $Q\in\mathcal{Q}_n$. By the previous lemma, we know that the support of $p(T\otimes I)(b_d^Q)p$ is included in $\X\setminus B_{3d_Q}(y_Q)$, so that we can write
\begin{align*}
\|p(T\otimes I)(b_d^Q)p\|_1&=\int_{\X\setminus B_{3d_Q}(y_Q)}\|p(x)[(T\otimes I)(b_d^Q)](x)p(x)\|_1dx\\
&\pp\int_{B_{3d_{\hat{Q}}}(y_{\hat{Q}})\setminus B_{3d_Q}(y_Q)}\|(T\otimes I)(b_d^Q)(x)\|_1dx+\int_{\X\setminus B_{3d_{\hat{Q}}}(y_{\hat{Q}})}\|(T\otimes I)(b_d^Q)(x)\|_1dx\\
&\lesssim\int_Q\|e_Qf(x)e_Q\|_1dx+\|b_d^Q\|_1\\
&=\sigma((1_Q\otimes e_Q)f(1_Q\otimes e_Q))+\|b_d^Q\|_1.
\end{align*}
Thus
\begin{align*}
    \sum_{n\pg1}\sum_{Q\in\mathcal{Q}_n}\|p(T\otimes I)(b_d^Q)p\|_1&\lesssim\sum_{n\pg1}\sum_{Q\in\mathcal{Q}_n}\sigma((1_Q\otimes e_Q)f(1_Q\otimes e_Q))+\sum_{n\pg1}\sum_{Q\in\mathcal{Q}_n}\|b_d^Q\|_1\\
    &=\sigma(efe)+\sum_{n\pg1}\sum_{Q\in\mathcal{Q}_n}\|b_d^Q\|_1\pp3\|f\|_1.
\end{align*}
As we know that $b_d=\sum_{n\pg1}\sum_{Q\in\mathcal{Q}_n}b_d^Q$ in $L_2(N)$, we deduce that we also have $p(T\otimes I)(b_d)p=\sum_{n\pg1}\sum_{Q\in\mathcal{Q}_n}p(T\otimes I)(b_d^Q)p$ in $L_2(N)$. By the previous estimate it follows that we actually have $p(T\otimes I)(b_d)p=\sum_{n\pg1}\sum_{Q\in\mathcal{Q}_n}p(T\otimes I)(b_d^Q)p$ in $(L_1\cap L_2)(N)$ with
\[\|p(T\otimes I)(b_d)p\|_1\lesssim3\|f\|_1.\]
The proof is complete.
\end{proof}

\begin{mdframed}[skipabove=10pt]
\begin{theo}\label{Cadhilac_estimate_off}
We have the estimate
\[\|pT(b_o)p\|_1\lesssim\|f\|_1.\]
\end{theo}
\end{mdframed}

\vspace{5pt}

For every $n\pg1$ and $Q\in\mathcal{Q}_n$, we set 
\[b_o^Q:=(1_Q\otimes q_Q)(f-1\otimes f_Q)(1_Q\otimes e_Q)+(1_Q\otimes e_Q)(f-1\otimes f_Q)(1_Q\otimes q_Q).\]

\begin{lemm}
We have
\[b_o=\sum_{n\pg1}\sum_{Q\in\mathcal{Q}_n}b_o^Q\ \ \ \ \ \ \ \ \text{in}\ L_2(N).\]
\end{lemm}
\begin{proof}
An easy computation shows that 
\[b_o=\sum_{n\pg1}\big[e_nfq_n+q_nfe_n\big]\ \ \ \ \ \ \ \ \text{in}\ L_2(N).\]
Moreover, by the commutation properties of Cuculescu's projections, we have
\[e_nE_n(f)q_n=q_nE_n(f)e_n=0\]
so that 
\[e_nfq_n+q_nfe_n=e_n(f-E_n(f))q_n+q_n(f-E_n(f))e_n.\]
The lemma directly follows.
\end{proof}

\begin{lemm}
Let $n\pg1$ and $Q\in\mathcal{Q}_n$. Then 
\[\int_Qq_Qf(y)q_Qdy\pp s|Q|.\]
\end{lemm}
\begin{proof}
By the properties of Cuculescu's projections, we have $q_nE_n(f)q_n\pp s$, and thus
\begin{align*}
    E_n((1_Q\otimes q_Q)f(1_Q\otimes q_Q))&=(1_Q\otimes q_Q)E_n(f)(1_Q\otimes q_Q)\\
    &=(1_Q\otimes 1)q_nE_n(f)q_n(1_Q\otimes 1)\pp s(1_Q\otimes 1).
\end{align*}
As the Bochner integral is invariant under $E_n$, we deduce that
\[\int_Qq_Qf(y)q_Qdy\pp s\int_Qdy=s|Q|.\]
\end{proof}

\begin{lemm}
Let $n\pg1$ and $Q\in\mathcal{Q}_n$. Then, for $x\in\X\setminus Q$ we have
\[\|(T\otimes I)(b_o^Q)(x)\|_1\lesssim s^{1/2}|Q|^{1/2}\Big\|\Big(\int_{Q}\big|k(x,y)-k(x,y_Q)|^2e_Qf(y)e_Qdy\Big)^{1/2}\Big\|_1.\]
\end{lemm}
\begin{proof}
For simplicity, we assume
\[b_o^Q=(1_Q\otimes e_Q)(f-1\otimes f_Q)(1_Q\otimes q_Q).\]
As the support of $b_o^Q$ is included in $Q$, and since $\int_\X b_d^Q(y)dy=0$, for $x\in\X\setminus Q$ we can use the integral representation to write
\begin{align*}
(T\otimes I)(b_o^Q)(x)&=\int_\X k(x,y)b_o^Q(y)dy\\
&=\int_{\X}\big[k(x,y)-k(x,y_Q)\big]b_o^Q(y)dy\\
&=\int_{Q}\big[k(x,y)-k(x,y_Q)\big]e_Q(f(y)-f_Q)q_Qdy\\
&=\int_{Q}\big[k(x,y)-k(x,y_Q)\big]e_Qf(y)q_Qdy-\int_{Q}\big[k(x,y)-k(x,y_Q)\big]e_Qf_Qq_Qdy.\\
\end{align*}
These two integrals can be estimated in a similar fashion. Thus, we can assume that
\[(T\otimes I)(b_o^Q)(x)=\int_{Q}\big[k(x,y)-k(x,y_Q)\big]e_Qf(y)q_Qdy\]
Then, by the Kadison-Schwarz inequality we obtain
\begin{align*}
\|(T\otimes I)(b_o^Q)(x)\|_1&\pp \Big\|\Big(\int_{Q}\big|k(x,y)-k(x,y_Q)\big|^2e_Qf(y)e_Qdy\Big)^{1/2}\Big\|_1\Big\|\int_{Q}q_Qf(y)q_Qdy\Big\|_\infty^{1/2}\\
&\pp s^{1/2}|Q|^{1/2}\Big\|\Big(\int_{Q}\big|k(x,y)-k(x,y_Q)|^2e_Qf(y)e_Qdy\Big)^{1/2}\Big\|_1
\end{align*}
where the last estimate comes from the previous lemma.
\end{proof}

\begin{lemm}
Let $n\pg1$ and $\mathcal{Q}\in\mathcal{Q}_n$. Then the support of $p(T\otimes I)(b_o^Q)p$ is included in $\X\setminus B_{3d_Q}(y_Q)$. 
\end{lemm}
\begin{proof}
It suffices to mimic the proof of Lemma \ref{Cadhilac_lemma_support}.
\end{proof}

\begin{lemm}
Let $n\pg1$ and $Q\in\mathcal{Q}_n$. Then 
\[\|p(T\otimes I)(b_o^Q)p\|_1\lesssim s^{1/2}\sigma(1_Q\otimes e_Q)^{1/2}\sigma(1_Q\otimes e_Qf_Qe_Q)^{1/2}.\]
\end{lemm}
\begin{proof}
One the one hand, from the two previous lemmas, we can write
\begin{align*}
&\|p(T\otimes I)(b_o^Q)p\|_1\pp\int_{\X\setminus B_{3d_Q}(y_Q)}\|(T\otimes I)(b_o^Q)(x)\|_1dx\\
&\pp s^{1/2}|Q|^{1/2}\int_{\X\setminus B_{3d_Q}(y_Q)}\Big\|\Big(\int_{Q}\big|k(x,y)-k(x,y_Q)|^2e_Qf(y)e_Qdy\Big)^{1/2}\Big\|_1dx\\
&= s^{1/2}|Q|^{1/2}\int_{\X\setminus B_{3d_Q}(y_Q)}\Big\|e_Q\Big(\int_{Q}\big|k(x,y)-k(x,y_Q)|^2e_Qf(y)e_Qdy\Big)^{1/2}e_Q\Big\|_1dx\\
&\pp s^{1/2}|Q|^{1/2}\tau(e_Q)^{1/2}\int_{\X\setminus B_{3d_Q}(y_Q)}\Big\|\Big(\int_{Q}\big|k(x,y)-k(x,y_Q)|^2e_Qf(y)e_Qdy\Big)^{1/2}\Big\|_2dx\\
&=s^{1/2}|Q|^{1/2}\tau(e_Q)^{1/2}\int_{\X\setminus B_{3d_Q}(y_Q)}\Big[\int_{Q}\big|k(x,y)-k(x,y_Q)|^2\tau(e_Qf(y)e_Q)dy\Big]^{1/2}dx\\
&\pp s^{1/2}|Q|^{1/2}\tau(e_Q)^{1/2}\sum_{m=0}^{\infty}\int_{B_{Q,m}}\Big[\int_{Q}\big|k(x,y)-k(x,y_Q)|^2\tau(e_Qf(y)e_Q)dy\Big]^{1/2}dx\\
\end{align*}
where $B_{Q,m}:=B_{2^{m+1}.3d_Q}(y_Q)\setminus B_{2^m.3d_Q}(y_Q)$. One the other hand, by H\"older's inequality, for $m\pg0$ we can write
\begin{align*}
    &\int_{B_{Q,m}}\Big[\int_{Q}\big|k(x,y)-k(x,y_Q)|^2\tau(e_Qf(y)e_Q)dy\Big]^{1/2}dx\\
&\pp|B_{Q,m}|^{1/2}\Big[\int_{B_{Q,m}}\Big[\int_{Q}\big|k(x,y)-k(x,y_Q)|^2\tau(e_Qf(y)e_Q)dy\Big]dx\Big]^{1/2}\\
&=|B_{Q,m}|^{1/2}\Big[\int_{Q}\Big[\int_{B_{Q,m}}\big|k(x,y)-k(x,y_Q)|^2dx\Big]\tau(e_Qf(y)e_Q)dy\Big]^{1/2}\\
&\pp|B_{Q,m}|^{1/2}\sup_{y\in Q}\Big[\int_{B_{Q,m}}\big|k(x,y)-k(x,y_Q)|^2dx\Big]^{1/2}\Big[\int_{Q}\tau(e_Qf(y)e_Q)dy\Big]^{1/2}\\
&\pp|B_{Q,m}|^{1/2}\sup_{y\in B_{d_Q}(y_Q)}\Big[\int_{B_{Q,m}}\big|k(x,y)-k(x,y_Q)|^2dx\Big]^{1/2}\Big[\int_{Q}\tau(e_Qf(y)e_Q)dy\Big]^{1/2}.
\end{align*}
Thus, by the $L_2$-H\"ormander regularity condition, we obtain the estimate
\[\|p(T\otimes I)(b_o^Q)p\|_1\lesssim s^{1/2}|Q|^{1/2}\tau(e_Q)^{1/2}\Big[\tau\int_{Q}e_Qf(y)e_Qdy\Big]^{1/2}.\]
Finally, as 
\[\int_Qe_Qf(y)e_Qdy=e_Q\big[\int_Q f(y)dy\Big]e_Q=|Q|e_Qf_Qe_Q\]
we get the desired result.
\end{proof}

\begin{proof}[Proof of Theorem \ref{Cadhilac_estimate_off}]
By the previous lemma, we have
\begin{align*}
\sum_{n\pg1}\sum_{Q\in\mathcal{Q}_n}\|p(T\otimes I)(b_o^Q)p\|_1&\lesssim s^{1/2}\sum_{n\pg1}\sum_{Q\in\mathcal{Q}_n}\sigma(1_Q\otimes e_Q)^{1/2}\sigma(1_Q\otimes e_Qy_Qe_Q)^{1/2}\\
&\pp s^{1/2}\Big[\sum_{n\pg1}\sum_{Q\in\mathcal{Q}_n}\sigma(1_Q\otimes e_Q)\Big]^{1/2}\Big[\sum_{n\pg1}\sum_{Q\in\mathcal{Q}_n}\sigma(1_Q\otimes e_Qy_Qe_Q)\Big]^{1/2}.
\end{align*}
But
\begin{align*}
   \sum_{n\pg1}\sum_{Q\in\mathcal{Q}_n}\sigma(1_Q\otimes e_Qy_Qe_Q)&=\sum_{n\pg1}\sum_{Q\in\mathcal{Q}_n}\sigma((1_Q\otimes e_Q)E_n(f)(1_Q\otimes e_Q))\\
   &=\sum_{n\pg1}\sum_{Q\in\mathcal{Q}_n}\sigma((1_Q\otimes e_Q)f(1_Q\otimes e_Q))\\
   &=\sum_{n\pg1}\sum_{Q\in\mathcal{Q}_n}\sigma((1_Q\otimes e_Q)f(1_Q\otimes e_Q))\\
   &\pp\sigma(efe)\pp\|f\|_1
\end{align*}
and 
\[\sum_{n\pg1}\sum_{Q\in\mathcal{Q}_n}\sigma(1_Q\otimes e_Q)=\sigma\Big(\sum_{Q\in\mathcal{Q}}1_Q\otimes e_Q\Big)=\sigma(e)=\sigma(1-q)\pp s^{-1}\|f\|_1\]
where we used the well-known properties of Cuculescu's construction. Thus, we get
\[\sum_{n\pg1}\sum_{Q\in\mathcal{Q}_n}\|p(T\otimes I)(b_o^Q)p\|_1\lesssim\|f\|_1.\]
As we have $b_o=\sum_{n\pg1}\sum_{Q\in\mathcal{Q}_n}b_o^Q$ in $L_2(N)$, we have $p(T\otimes I)(b_o)p=\sum_{n\pg1}\sum_{Q\in\mathcal{Q}_n}p(T\otimes I)(b_o^Q)p$ in $L_2(N)$. By the previous estimate we deduce that we actually have $p(T\otimes I)(b_o)p=\sum_{n\pg1}\sum_{Q\in\mathcal{Q}_n}p(T\otimes I)(b_o^Q)p$ in $(L_1\cap L_2)(N)$ with
\[\|p(T\otimes I)(b_o)p\|_1\lesssim3\|f\|_1.\]
\end{proof}

\section*{\phantom{text}}

\section*{Acknowledgements}

\vspace{4pt}

I am very grateful to my advisor Éric Ricard for many valuable discussions and his guidance throughout the writing of this article.

\bibliographystyle{plain}
\bibliography{BIBLIOGRAPHY} 

\end{document}

%% file: PREAMBULE.tex
\usepackage{amsmath,amssymb,amsthm,graphicx,url}

\usepackage[bookmarksnumbered, colorlinks, plainpages]{hyperref}
\hypersetup{colorlinks=true,linkcolor=blue, anchorcolor=green, citecolor=violet, urlcolor=cyan, filecolor=magenta, pdftoolbar=true}

\usepackage[english]{babel}
\usepackage[T1]{fontenc}
\usepackage{color}
\usepackage{standalone}
\usepackage{enumitem}
\usepackage{tikz-cd}
\usepackage{fbox}
\usepackage{mdframed}
\usepackage[shortcuts]{extdash}

\usepackage[a4paper, left=2.8cm, right=2.8cm, top=3.5cm, bottom=3.5cm]{geometry}

\newtheorem{maintheo}{Theorem}

\newtheorem{lemm}[theo]{Lemma}

\newtheorem{prop}[theo]{Proposition}
\newtheorem{coro}[theo]{Corollary}

\theoremstyle{definition}
\newtheorem{defi}[theo]{Definition}

\theoremstyle{remark}
\newtheorem*{rem}{Remark} 

\numberwithin{equation}{section}

\def\C{{\mathbb C}}

\def\R{{\mathbb R}}
\def\S{{\mathbb S}}
\def\T{{\mathbb T}}

\def\X{{\mathbb X}}
\def\Z{{\mathbb Z}}

\def\TT{{\mathbb T}}

\def\pp{\leq}
\def\pg{\geq}
\def\bra{\langle}
\def\ket{\rangle}

\newcommand{\fonctbis}[4]{\left\{\begin{array}{ccc} #1&\rightarrow&#2 \\ #3&\mapsto&#4 \end{array}\right.}

\DeclareMathOperator{\ima}{ran}

\DeclareMathOperator{\re}{Re}

\DeclareMathOperator{\diam}{diam}
\DeclareMathAlphabet\cat{OMS}{cmsy}{b}{n}

\usepackage{newtxtext,newtxmath}

\makeatletter

\providecommand{\proofname}{Preuve}
\makeatother